\documentclass[11pt]{amsart}
\usepackage[T1]{fontenc}
\usepackage[ansinew]{inputenc}
\usepackage{amssymb}
\usepackage[usenames]{color}
\newtheorem{thm}{Theorem}[section]

\newtheorem{mthm}{Main result}
\newtheorem{corollary}[thm]{Corollary}
\newtheorem{lem}[thm]{Lemma}

\newtheorem{proposition}[thm]{Proposition}
\newtheorem{definition}[thm]{Definition}
\newtheorem{remark}{Remark}[section]

\newcommand{\eps}{\varepsilon}

\begin{document}
\title[Controllability of the NLS via a low-dimensional source term]{Controllability of the cubic Schroedinger equation via a low-dimen\-sional source term}%
\author{Andrey Sarychev}
\address{DiMaD,
University of Florence, Italy}%
\email{asarychev@unifi.it}%
\date{}%
\begin{abstract}
We  study controllability of $d$-dimensional defocusing cubic Schroe\-din\-ger equation under periodic boundary conditions. The control is applied additively, via a source term, which    is a  linear combination
of few  complex exponentials (modes) with  time-variant coefficients - controls. We manage to prove that controlling at most $2^d$  modes one can achieve controllability of the equation in any finite-dimensional projection of the  evolution space $H^{s}(\mathbb{T}^d), \ s>d/2$,  as well as  approximate controllability in $H^{s}(\mathbb{T}^d)$. We also present  negative result regarding exact controllability of cubic Schroedinger equation via a finite-dimensional source term.
\end{abstract}
\maketitle


\vspace*{2mm}

\section{Introduction}
\label{intro} 

Lie algebraic approach of geometric control theory to the {\it nonlinear distributed systems}  has been initiated recently. An example of its implementation is  the  study of $2$-dimensional Navier-Stokes/Euler equations of fluid motion
controlled by {\it low-dimensional forcing} in  \cite{AS06,AS08}, where
one arranged sufficient criteria for approximate controllability and for controllability in finite-dimen\-si\-onal projections of  the evolution space.

Here we wish to develop similar approach for another  class of distributed system, which is  cubic defocusing Schroedinger equation  (cubic defocusing NLS):
\begin{equation}
-i \partial_t u(t,x) +  \Delta u(t,x) =|u(t,x)|^2u(t,x) +V(t,x), \ u|_{t=0}=u^0, \label{nls1}
\end{equation}
controlled via a  source term $V(t,x)$.

We treat the periodic case: the space variable $x$ belongs   to the   torus $\mathbb{T}^d$.

Our problem setting is distinguished by  two features.
First,  the control is introduced via a source term, i.e. in additive form, on the contrast to the bilinear
form, characteristic for quantum control models. Second  feature is
{\it finite-di\-men\-si\-onality of the range of the controlled source term}:
\begin{equation}\label{sote0}
V(t,x)=\sum_{k \in \hat{\mathcal K}}v_k(t)e^{ik \cdot x},
\end{equation}
 with  $\hat{\mathcal K} \subset \mathbb{Z}^d$ being   a finite set. Obviously  for each $t$ the value $V(t, \cdot)$ belongs to a   {\it finite-di\-men\-si\-onal subspace} $\mathcal{F}_{\hat{\mathcal K}}=\mbox{Span}\{e^{i k \cdot x}, \ k \in \hat{\mathcal K}\}$ of the evolution space for NLS.

The {\it control functions} $v_k(\cdot), \ k \in \hat{\mathcal K}$, which enter  the source term,
 can be chosen freely in ${\mathbf L}^\infty[0,T]$, or in any functional space, which
is dense in ${\mathbf L}^1[0,T]$.

    By the choice of 'low-di\-men\-si\-onal' control  our problem setting differs from the studies of controllability of NLS (see the end of the Section~\ref{surv} for few references to the alternative settings and approaches), in which controls have  infinite-di\-men\-si\-onal range.
       In some of the studies the controls  have their  support on a subdomain and one is interested in tracing the propagation of the controlled energy to other parts of the domain. On the contrast, in our case the controls affect few modes - 'directions' -  in  the functional evolution space for NLS and we are in\-te\-rested in the way this controlled action spreads to other (higher) modes.

 We will treat   NLS equation (\ref{nls1}) as  an evolution equation in  Sobolev space $H^{s}(\mathbb{T}^d), \ s>d/2$.
 By opting  for  'higher regularity' of solutions with respect to the  space variables we seek to  avoid analytic complications (see Remark~\ref{stri}), which are not directly related to the controllability issue under study.

  Choosing   initial condition $u(0)=u_0 \in H^{s}(\mathbb{T}^d),$ we set the problems of:
\begin{enumerate}
    \item  controllability in finite-di\-men\-si\-onal projections, meaning that one can steer  in time $T>0$ the  trajectory of the equation (\ref{nls1}) from  $u_0$ to a state $\hat{u} \in H^{s}$
        with any preassigned orthogonal
         \footnote{In the formulation of results below and in the course of the proofs we speak of  orthogonal projections on finite-di\-men\-si\-onal subspaces. The formulations and proofs are valid for orthogonality coming from  either $L^2$ inner product, or $H^s$ inner product. Note that the complex exponentials form orthogonal system with respect to both products.}
                  projection $\Pi^{\mathcal{L}}\hat{u}$ onto a  given finite-di\-men\-si\-onal subspace $\mathcal{L} \subset H^{s}$;
  \item  approximate controllability meaning that the attainable set of (\ref{nls1}) from each $u_0$ is dense in $H^{s}$;
  \item exact controllability in $H^{s}$.
    \end{enumerate}

Definitions of several  types of controllability and exact problem setting are provided in the next Section together with the main results. First of the results asserts  that controllability  in projection on each finite-di\-men\-si\-onal subspace of $H^{s}$ and approximate controllability in $H^{s}$ can be achieved by applying controls to at most  $2^d$ modes, which can be chosen  the same for all subspaces. Propositions~\ref{z1},\ref{satu} describe  classes of saturating sets of controlled modes which suffice for achieving these types of controllability. The second main result asserts lack of exact controllability in $H^{s}$ by controlling any finite number of modes.

The author is grateful to two anonymous referees whose thorough refereeing and valuable remarks
helped to improve the presentation in the revised version.

\section{Cubic Schroedinger equation on $\mathbb{T}^d$; problem setting and main results} \label{schr}

\subsection{Controllability: definitions}\label{codef}

\subsubsection{Global controllability}
As we mentioned before the evolution space of NLS equation is
Sobolev space $H=H^{s}(\mathbb{T}^d), s>d/2$.

We say that   the  controlling source term  (\ref{sote0})
{\it steers the system (\ref{nls1})
from $u_0 \in H$ to $\hat{u} \in H$ in
time $T>0$}, if mild (see Section~\ref{prel}) solution of (\ref{nls1})
exists, is unique
and satisfies $u(T)=\hat{u}$. The set $\mathcal{A}_{T,u^0}$ of points
to which the system can be steered to from $u^0$ in time $T>0$ is called the {\it attainable set}.

The equation is {\it globally time-$T$ (exactly) controllable} from $u_0$,
if it can be steered in time $T$ from $u_0$ to any point of $H$, or, the same, $\mathcal{A}_{T,u^0}=H$.


\subsubsection{Controllability in finite-di\-men\-si\-onal projections and in finite-dimensio\-nal components}

Let $\mathcal L$ be  closed linear subspace of $H$, $\Pi^{\mathcal L}$ be the orthogonal projection of $H$ onto ${\mathcal L}$.

The equation (\ref{nls1})-(\ref{sote0}) is time-$T$ globally {\it controllable  in projection
on $\mathcal L$}, if for each $\hat{q} \in \mathcal L$ this equation can be steered in time $T$  to some point $\hat{u}$ with $ \Pi^{\mathcal L}(\hat{u})=\hat{q}$.

The NLS equation (\ref{nls1})-(\ref{sote0}) is time-$T$ globally {\it controllable  in finite-dimen\-si\-onal projections} if for each finite-di\-men\-si\-onal subspace $\mathcal L$ the equation is time-$T$ globally controllable  in projection
on $\mathcal L$;  the set $\hat{\mathcal K}$ of controlled modes is assumed to be the same for all $\mathcal L$.

Whenever $\mathcal{L}$ is a 'coordinate subspace' $\mathcal{L}=\mathcal{F}_{\mathcal{K}^{o}}=\mbox{Span}\{e^{ik\cdot x}| \ k \in \mathcal{K}^{o}\}$,
with a finite set $\mathcal{K}^{o} \subset \mathbb{Z}^d$  of {\em observed modes}, then controllability in projection on $\mathcal{L}$ is called {\it controllability in observed $\mathcal{K}^{o}$-component}.

\begin{remark}\label{pil}
It is convenient to characterize the
 time-$T$ controllability    in terms of   surjectiveness of the
{\it end-point map}  $E_T : v(\cdot)  \to u(T)$,
of  the controlled  NLS equation (\ref{nls1})-(\ref{sote0}).  The map  $E_T$    maps
 a control $v(\cdot)=\{v_k(\cdot)| k \in \hat{\mathcal K}\},$  into the end-point $u(T)$ of the
trajectory $u(t)$ of (\ref{nls1}),  driven by the source term (\ref{sote0}). Equivalently one can  see the end-point map, as mapping
the controlling source term $V=\sum_{k \in \hat{\mathcal K}}v_k(t)e^{ik \cdot x}, \ t \in [0,T]$
  to $u(T)$.

The time-$T$ controllability  in projection on $\mathcal{L}$ means surjectiveness  of the composition
 $\Pi_{\mathcal{L}} \circ E_T. \ \Box$
\end{remark}

\subsubsection{Approximate controllability}

The NLS equation  (\ref{nls1})-(\ref{sote0}) is time-$T$ approximately controllable from $u_0$ in $H$, if it can be steered to each point
of a dense subset of $H. \ \Box$

\subsubsection{Solid controllability (cf. \cite{AS08})}
\label{soc}
On the contrast to  the previous definitions the word 'solid' does not refer to a new  type of controllability
but means  stability of the controllability
 property with respect to certain classes of perturbations.

Let $\Phi: \mathcal{M}^1 \mapsto \mathcal{M}^2$ be a continuous map
between two metric spaces, and $S \subseteq \mathcal{M}^2$ be any
subset. We say that {\it $\Phi$ covers $S$ solidly}, if $\Phi(\mathcal{M}^1) \supseteq S$ and the inclusion is stable with respect to
$C^0$-small perturbations of $\Phi$, i.e. for some
$C^0$-neighborhood $\Omega$ of $\Phi$ and for each map $\Psi \in
\Omega$, the inclusion  $\Psi(\mathcal{M}^1) \supseteq S$ holds.

Given a finite-di\-men\-si\-onal subspace $\mathcal{L} \subset H^s$ and a bounded set $S \subset \mathcal{L}$  we  say that $S$ is solidly attainable  for (\ref{nls1})-(\ref{sote0})  in time $T$, if
there exists a family  of controls $\{v(\cdot, b)| \  b \in B_S\}, \ v(\cdot ,b)=\{v_k(\cdot,b)| k \in \hat{\mathcal K}\}$ parameterized by a compact set $B_S \subset \mathbb{R}^N$,
such that  the map
$b \mapsto \left(\Pi^{\mathcal L}
\circ E_T\right)(v(\cdot,b))$ is continuous on $B_S$ and covers $S$ solidly.

 Time-$T$ solid controllability in projection on finite-di\-men\-si\-onal subspace $\mathcal{L}$
for the NLS equation (\ref{nls1})-(\ref{sote0})  means that any bounded  subset $S \subset \mathcal L$
is solidly attainable in time $T$.

\subsection{Main results}

       Our first goal is establishing   for the cubic defocusing NLS  on $\mathbb{T}^d$  sufficient criteria of controllability in  finite-di\-men\-si\-onal projections and of approximate controllability in Sobolev space $H^{s}, \ s>d/2$. The common criterion is  formulated in terms of   the set of controlled modes $\hat{\mathcal K}$, which  is fixed and the same for {\it all} projections and for the approximate controllability.

\begin{mthm}[sufficient criterion for controllability in finite-di\-men\-si\-onal
projections and for approximate controllability]
\label{mainth}
For each integer $d \geq 1$ there exists a $2^d$-element set
$\hat{\mathcal K}$  of
controlled modes  ($\hat{\mathcal K} \subset \mathbb{Z}^d$) such that
 for each $s> d/2$, any  initial
data  $u^0 \in H^{s}(\mathbb{T}^d)$ and any $T>0$ the defocusing cubic Schroe\-din\-ger equation (\ref{nls1}) on $\mathbb{T}^d$,  controlled via the source term
(\ref{sote0}) is:  i)   time-$T$
controllable in  projection on each finite-di\-men\-si\-onal subspace   $\mathcal{L}$ of $H^s(\mathbb{T}^d)$; ii) time-$T$ approximately controllable in $H^{s}(\mathbb{T}^d). \ \Box$
\end{mthm}

Our second result  is

\begin{mthm}[lack of  exact controllability via a finite-di\-men\-si\-onal additive control]
\label{2main}
Let   the defocusing cubic Schroe\-din\-ger equation (\ref{nls1})  on $\mathbb{T}^d, \ d \geq 1,$ be  controlled via a source term
(\ref{sote0}), with the  set
$\hat{\mathcal K} \subset \mathbb{Z}^d$ of controlled modes being arbitrary finite.  Then for each $s>d/2$, each  $T>0$ and each initial
 data  $u^0 \in H^s(\mathbb{T}^d)$, the time-$T$ attainable set $\mathcal{A}_{T,u^0}$
 of (\ref{nls1})-(\ref{sote0})  is contained in a countable union of compact subsets of $H^s(\mathbb{T}^d)$ and hence its  complement
$H^s(\mathbb{T}^d) \setminus  \mathcal{A}_{T,u^0}$ is dense in $H^s(\mathbb{T}^d). \ \Box$
\end{mthm}

\section{Outline of the approach: Lie extensions, fast-oscillating controls, resonances. Other approaches}
\label{surv}

  Study of controllability of  the NLS equation is based  (as well as our earlier work \cite{AS06},\cite{AS08} on Navier-Stokes/Euler equation)  on the method of iterated Lie extensions.
  {\it Lie extension} of  a control system $\dot{x}=f(x,u), \ u \in U$ is a way to add vector fields to the right-hand side of the system, maintaining at the same time (almost) invariance of its controllability properties. The additional vector fields are expressed via the Lie brackets of the vector fields $f(\cdot ,u)$ for various $u \in U$.
  If after a series of extensions one arrives to a controllable (extended) system, then the controllability of the original system would follow.

This approach can not be propagated  automatically to the  infinite-di\-men\-si\-onal setting due to the
lack of adequate  Lie algebraic tools.  So far  Lie algebraic viewpoint in the infinite-di\-men\-si\-onal  context  is  mainly  used    as a guiding tool; the   implementation has to  be  justified by the analytic means. In the rest of this Section we provide geometric control sketch for  the proof of main result.

When studying controllability we will look at the  cubic NLS equation as at a  particular type of (infinite-di\-men\-si\-onal) {\it control-affine system}:
$$-i\partial_t u=c(u,t)+\sum_{k \in \hat{\mathcal K}}e_k v_k(t),$$
where $c(u,t)$ is the cubic {\it drift} vector field, $e_k$ are  constant  (independent of $u$)  {\it controlled} vector field in $H^s(\mathbb{T}^d)$ with  the values $e^{ik \cdot x}$.

The Lie extensions, we   implement,   involve two  controlled vector fields $e_m,e_n$ and the drift vector field $c$. An outcome of each extension is  the fourth-order  Lie bracket
$ [e_n,[e_m,[e_m,c]]]$, which is again constant vector field (given that the vector field $c$ is cubic). The Lie bracket   is seen as  a direction  of  the action of an {\it extended control}.

Another Lie bracket, which makes its appearance  at  each Lie extension,  is
 the third-order Lie bracket $[e_m,[e_m,c]]$.  The motion of the system along its direction is 'unilateral'; hence it can be seen as  {\it obstruction to controllability}. This motion  can not be locally reverted nor compensated but for  the NLS equation  one can nullify it  {\it in average} by imposing integral (isoperimetric) relations onto the controls involved in the extension.

To design the needed motion in the extending direction $[e_n,[e_m,[e_m,c]]]$ and to oppress the obstructing motions, we  employ  {\it fast-oscillating controls}. Use of such controls is traditional for geometric control theory.  A  'general theory' of their application is hardly available,
but particular cases can be effectively studied (see, for example \cite{SuL}).

           In our study we invoke for each extension a couple of  fast-oscillating controls
           $$v_m(t)e^{i a_mt/\eps}e_m,v_n(t)e^{i a_nt/\eps}e_n,$$
           feed them into the right-hand side of the NLS equation and trace the  interaction of the two controls  via the cubic term of the equation.   The idea is to design needed resonance in the course of the interaction, that is to choose  oscillation frequences  and magnitudes of controls in such a way that the {\it averaged interaction}   influences dynamics of  specific  modes.

           By our design we manage to limit the influence of the resonance\footnote{Other choices of resonant modes are possible, being this one one of simplest}  to a single  mode  $e_{2m-n}=e^{i(2m-n) \cdot x}$;   the resonance term will be  seen as an (extending) control for   this mode. The procedure  can be  interpreted as  an {\it elementary extension} of the set of controlled modes: $ \hat{\mathcal K} \mapsto \hat{\mathcal K} \bigcup \{2m-n\},  \ m,n \in \hat{\mathcal K}.$

 Final controllability result (Main result 1) is obtained by (finite) iteration of the    elementary  extensions. If one seeks controllability in  observed $\mathcal{K}^o$-component with $\mathcal{K}^o \supset \hat{\mathcal K}$, then one should find (when possible)  a series of elementary extensions  $\hat{\mathcal K} = \mathcal{K}^1 \subset \mathcal{K}^2 \subset \cdots \subset \mathcal{K}^N=\mathcal{K}^o$. Getting extended controls  available for each observed mode $k \in \mathcal{K}^o$,  we may conclude  {\it controllability of  extended system in $\mathcal{K}^o$-component} by    Lemma~\ref{m=1}.
On the contrast controllability of the original system  in $\mathcal{K}^o$-component will follow by virtue of  technical Approximative Lemma~\ref{appex}, which formalizes the resonance design.

From the controllability in  {\it all} finite-di\-men\-si\-onal components  one derives  controllability in projection on each finite-di\-men\-si\-onal subspace  as well as  the approximate controllability; this  is proved in  Section~\ref{proapc}.

Note that in the case of {\it periodic} NLS equation the analysis of interaction of different terms via cubic nonlinearity  is substantially simplified by  our choice of a special basis of  exponential  modes.

   Besides the design of   proper resonances there are two {\it analytic}  problems to be fixed. First problem consists of studying NLS with fast-oscillating right-hand sides  and of establishing the continuity, the approximating properties and   the {\it limits} of corresponding trajectories, as the frequency of oscillation tends to infinity.  Second problem is  coping  with the fact that at each iteration we are only able to {\it approximate} the desired motion, therefore the controllability criteria need to be  stable with respect to the approximation errors.

The second problem is fixed with the help of the notion of {\it solid
controllability} (see previous Section).

The solution to the first problem  in finite-di\-men\-si\-onal setting   is provided  by
{\it theory of relaxed controls} (\cite{Gam}). For general nonlinear PDE such theory is unavailable;
   while for {\it semilinear infinite-di\-men\-si\-onal control systems} relaxation results have been obtained in \cite{Fr,Fat}. We provide formulations and proofs needed for our analysis in Subsection~\ref{corel}.

  What regards negative result on exact controllability (Main result~2), then the key point for its proof is endowing the (functional) space of controls with a weaker topology, in which this space is a countable union of compacts and  the end-point  map is Lipschitzian in the respective metric. Then the attainable  sets - the images of the end point map - are meager.
   Similar kind of argument has been used in \cite{BMS} for establishing noncontrollability of some bilinear distributed systems. Finer method, based on estimates of Kolmogorov's entropy has been invoked in \cite{Shi1} for proving lack of exact controllability by finite-di\-men\-si\-onal forcing for Euler equation of fluid dynamics.

At the end of the Section we wish to mention  few references to other approaches to controllability of linear and  semilinear Schroedinger equation controlled via  bilinear  or additive control, this latter being "internal"  or boundary.

References  \cite{Zu,ILT}  provide nice surveys of the results on:
\begin{itemize}
  \item exact controllability for linear Schroedinger equation with additive control in relation to observability of adjoint system and to geometric control condition (see also \cite{Le});
     \item controllability of linear Schroedinger equation with control entering bilinearly; besides references in the above cited surveys there are notable results \cite{Bea,BC} on local (exact) controllability in $H^7$ of 1-D equation; another interesting result is (obtained by geometric methods) criterion \cite{CMSB} of approximate controllability for the case in which 'drift Hamiltonian' has discrete non-resonant spectrum (see bibliographic references in \cite{Bea,BC,CMSB} to the preceding work);
           \item exact controllability of semilinear Schroedinger equation by means of internal additive control; in addition to references in \cite{Zu,ILT} we mention more recent publications \cite{DGL,RZ} where the property has been established for 2D and 1D cases.     The key tool in the study of the semilinear case is 'linearization principle', going back to \cite{LT}. On the  contrast our approach makes exclusive use of  nonlinear term.
\end{itemize}

\section{Preliminaries on existence, uniqueness and continuous dependence of trajectories}
\label{prel}

  In this Section we collect results on existence/uniqueness and  continuity in the right-hand side
   for the {\it mild} solutions of a Cauchy problem for a class of semilinear equations:
\begin{equation}\label{nlsmod}
(-i\partial_t + \Delta)u=G(t,u), \  u(0)=\tilde{u}^0.
\end{equation}
The mild solutions satisfy integral form of the equation \eqref{nlsmod}, obtained via Duhamel formula:
   $$u(t)=e^{it\Delta}\left(\tilde{u}^0+ i\int_0^te^{-i\tau \Delta}G(\tau, u(\tau))d\tau\right).$$

The mild solutions  $u(\cdot)$  are sought in the space $C([0,T];H)$, with $H$, being Hilbert space of functions $u(x)$ defined on $\mathbb{T}^d$. We opt for $H=H^{s}(\mathbb{T}^d), \ s>d/2$.

Accordingly a   'perturbation' of the Cauchy problem \eqref{nlsmod}:
\begin{equation}\label{nlsper}
(-i\partial_t + \Delta)u=G(t,u)+\phi(t,u), u(0)=u^0,
\end{equation}
 admits the integral form
\begin{equation}
   u(t)=e^{it\Delta}\left(u^0+ i\int_0^t e^{-i\tau \Delta} \left(G(\tau,u(\tau))+\phi(\tau,u(\tau))\right)d\tau\right), \label{iper}
\end{equation}
whose solutions are the mild solutions of \eqref{nlsper}.

We assume the nonlinear terms  $G(\cdot ,\cdot),\phi(\cdot  ,\cdot): \ [0,T] \times H \to H$ in \eqref{nlsmod}-\eqref{iper}
to
satisfy the conditions:
\begin{eqnarray}
G,\phi \ \mbox{are continuous;} \label{congfi} \\
 \forall c>0, \ \exists \beta_c(t) \in {\mathbf L}^1([0,T],\mathbb{R}_+), \ \mbox{such that} \
  \forall t \in [0,T], \ \forall \|u\| \leq c, \nonumber \\
 \label{gbip}  \|G(t,u)\|_H \leq \beta_c(t), \  \|G(t,u')-G(t,u)\|_H \leq \beta_c (t) \|u'-u\|_H,\\
 \label{fbip} \|\phi(t,u)\|_H \leq \beta_c(t),  \  \|\phi(t,u')-\phi(t,u)\|_H \leq \beta_c (t) \|u'-u\|_H .
 \end{eqnarray}

\begin{proposition}[local existence and uniqueness of mild solutions]
\label{loex}
Let $G$ satisfy  the conditions \eqref{congfi},(\ref{gbip}).
Then for each $c>0, \ \exists T_c >0$ such that   for $\|\tilde{u}^0\|_{H} \leq c$
there exists unique mild solution $u(\cdot) \in C([0,T_c], H)$
of the Cauchy problem   (\ref{nlsmod}).  $\Box$
\end{proposition}

The result is proved via fixed point argument for contracting map in $C([0,T];H)$.

 Our choice of  $H=H^s(\mathbb{T}^d), \ s>d/2$, guarantees that  the cubic term of the NLS equation (\ref{nls1}) would  satisfy conditions \eqref{congfi},(\ref{gbip})
 according to
 the following version of embedding theorem.

\begin{lem}['Product Lemma'; \cite{Tao}]\label{prol}
For Sobolev spaces $H^s(\mathbb{T}^d)$ of functions there holds:

for $s \geq 0: \  \|fg\|_{H^s} \leq C(s,d)\left(\|f\|_{H^s}\|g\|_{L^\infty}+\|f\|_{L^\infty}\|g\|_{H^s}\right) ;$

for $s> d/2: \ \|fg\|_{H^s} \le C'(s,d)\|f\|_{H^s}\|g\|_{H^s}. \ \Box  $
\end{lem}

This Lemma  allows  verification of the conditions \eqref{congfi},(\ref{gbip}),(\ref{fbip})  for more general
 classes of operators   of the form
$$u(x) \mapsto P(u(x),\bar{u}(x)), $$
where $P$ is  a polynomial in $u,\bar{u}$ with  coefficients $p_{jk}(t,x)$ such that
 $p_{jk}(\cdot,x)\in {\mathbf L}^1([0,T],\mathbb{C})$. Recall that  the source term (\ref{sote0})  is a
trigonometric polynomial in $x$ and $F(\cdot ,x) \in {\mathbf L}^\infty\left([0,T],H\right)$.

\begin{remark}\label{stri}
Let us remark on the obstacles, which may arise, when one considers solutions of the NLS with   lower regularity in space variables.
For proving local existence,
uniqueness and well-posedness of solutions of the NLS in $H^s$ with $s \leq d/2$,  one invokes Strichartz inequalities.  While for the NLS equation in $\mathbb{R}^n$  Strichartz inequalities  are  derived from dispersion estimates (\cite{Tao}), this approach would fail  on a compact manifold (\cite{BGT}), since there the dispersion estimates are not available.    Still kind of Strichartz inequalities  with the loss of derivatives can be established  on a compact Riemannian manifold; see   \cite{BGT} and references therein  to the previous work. For flat torus $\mathbb{T}^d$ Strichartz estimates have been derived in \cite{Bou} by methods of harmonic analysis.

For settling controllability issue in the low regularity  setting one would need results on continuity with respect to the right-hand side, analogous to   Propositions~\ref{wecodep},~\ref{wecopar}, and more important analogues of the results of Subsection~\ref{corel}, which justify application of relaxed (and fast-oscillating) controls to the NLS equation. We trust that all the mentioned obstacles can be overcome and  controllability criteria for more general setting  will appear in future publications. $\Box$
\end{remark}

{\it Global} existence results for a {\it cubic defocusing} NLS equation (\ref{nls1}) are classical under assumptions we made.

\begin{proposition}[global existence and uniqueness of mild solutions]
\label{gloex}
Let $H=H^s(\mathbb{T}^d), \ s>d/2$ and the  time-variant source term $t \mapsto F(t,\cdot)$  belong to
  ${\mathbf L}^1([0,T],H)$.
Then for each initial condition $u(0)=u^0 \in H $ the Cauchy problem   (\ref{nls1}) has a unique mild  solution $u(\cdot) \in C([0,T], H). \ \Box$
\end{proposition}

One can consult  \cite{DGL}, where such result is established  for cubic defocusing NLS with source term and weaker regularity of the data.

   Now we  provide few results on the  continuity of the solutions  of the NLS equation in the right-hand side. The topology in the space of right-hand sides is  introduced via the seminorms:
   $$\|\phi(t,u)\|^1_{T,c}=\int_0^T \sup_{\|u\|_H \leq c}\|\phi(t,u)\|_Hdt . $$

\begin{proposition}[continuity of the solutions with respect to  the right-hand side]
 \label{wecodep}
 Let $G,\phi$ satisfy  the conditions \eqref{congfi},(\ref{gbip}), \eqref{fbip} and $\tilde{u}(t) \in C([0,T],H)$ be a  mild  solution
 of  (\ref{nlsmod});  assume  $\sup_{t \in [0,T]}\|\tilde{u}(t)\|_H <c$.   Then $\exists \delta >0, C>0,$
such that whenever
\begin{equation}\label{dede}
\|u^0 - \tilde{u}^0\| +  \|\phi(t,u)\|^1_{T,c} < \delta ,
\end{equation}
  then  mild solution $u(t)$ of the perturbed equation (\ref{nlsper})
 exists on the interval $[0,T]$,  is unique   and admits  an estimate
 \begin{equation}\label{ubo}
\sup_{t \in [0,T]}\|u(t)-\tilde{u}(t)\| <C\left(\|u^0 - \tilde{u}^0\|+ \|\phi(t,u)\|^1_{T,c} \right). \ \Box
 \end{equation}
 \end{proposition}

\begin{proof}
 A solution of  the equation (\ref{nlsper}) can be continued in time as long as its norm in $H$ remains bounded.
Therefore one is able to conclude  extendibility of  the solution  onto $[0,T]$ from  the estimate (\ref{ubo}).

Estimating  the difference
\begin{eqnarray*}
  u(t)-\tilde{u}(t) = e^{it\Delta}\left((u^0-\tilde{u}^0)+ i\int_0^t e^{-i\tau \Delta}\phi(\tau,\tilde{u}(\tau))d\tau \right)+
  \\
  +e^{it\Delta}i\int_0^t e^{-i\tau \Delta}\left((G(\tau,u(\tau))-G(\tau,\tilde{u}(\tau)))+(\phi(\tau,u(\tau))-\phi(\tau,\tilde{u}(\tau)))\right)
  d\tau ,
\end{eqnarray*}
and noting  that $e^{it\Delta}$ is an isometry of $H$, we get
\begin{eqnarray}\label{umut}
  \|u(t)-\tilde{u}(t)\|_H \leq \|u^0-\tilde{u}^0\|_H+\left\|\int_0^t e^{-i\tau \Delta} \phi(\tau,\tilde{u}(\tau))d\tau \right\|_H+
  \\
  +\int_0^t \|e^{-i\tau \Delta}\left(G(\tau,u(\tau))-G(\tau,\tilde{u}(\tau))+\phi(\tau,u(\tau))-\phi(\tau,\tilde{u}(\tau))\right)\|_H
  d\tau \leq  \nonumber \\
   \leq \|u^0-\tilde{u}^0\|_H+ \left\|\int_0^t e^{-i\tau \Delta}\phi(\tau,\tilde{u}(\tau))d\tau \right\|_H
  +2\int_0^t \beta_c (\tau)\|u(\tau)-\tilde{u}(\tau)\|_Hd\tau  . \nonumber
\end{eqnarray}
    By the Gronwall inequality
\begin{eqnarray}\label{gro54}
\|u(t)-\tilde{u}(t)\|_H \leq \\
\leq \left(\|u^0-\tilde{u}^0\|_H+\left\|\int_0^t e^{-i\tau \Delta}\phi(\tau,\tilde{u}(\tau))d\tau \right\|_H\right)C'e^{C'\int_0^t \beta_c (\tau)d\tau}, \nonumber
\end{eqnarray}
for some $C'>0$ and whenever (\ref{dede}) is satisfied, we get for some $C$
\begin{equation}\label{55}
  \|u(t)-\tilde{u}(t)\|_H \leq  C\left(\|u^0-\tilde{u}^0\|+\int_0^t \left\|\phi(\tau,\tilde{u}(\tau)) \right\|d\tau \right) \leq   C \delta .
\end{equation}
One should choose $\delta >0$ such that $C\delta < c-  \sup_{t \in [0,T]}\|\tilde{u}(t)\|_H$.
\end{proof}

Next proposition is  parametric reformulation of the previous result.
\begin{proposition}[continuous dependence of the solutions on  parameter]
 \label{wecopar}
 Let a family of operators $G(t,u,b)$, parameterized by $b \in B \subset \mathbb{R}^N$: i)  be continuous in $b$ with respect to  each  seminorm $\|\cdot\|^1_{T,c}$; ii)  satisfy  the conditions (\ref{gbip})
 with the same $\beta_c(\cdot)$ for all $b \in B$.
 Then  the mild solutions of the equations
 \begin{equation}\label{parsem}
(-i\partial_t + \Delta)u=G(t,u,b), u(0)=u^0,
\end{equation}
 depend continuously on $b,u^0$ in the  $C^0$-metric.

   Besides if  $\tilde{u}(t) \in C([0,T],H)$ is   a mild  solution  of  (\ref{parsem}) for  $b=\tilde{b},\ u^0=\tilde{u}^0$, and   $\sup_{t \in [0,T]}\|u(t)\|_H dt <c$, then  $\exists \delta >0, C>0,$
such that whenever
$$\|u^0 - \tilde{u}^0\| + \|G(t,u,b)-G(t,u,\tilde{b})\|^1_{T,c} < \delta ,$$
    the  solutions   $u(t)$ of the equation (\ref{parsem})
 exist on the interval $[0,T]$, are  unique  and
\begin{eqnarray*}
  \sup_{t \in [0,T]}\|u(t)-\tilde{u}(t)\| \leq \\
  \leq C \left(\|u^0 - \tilde{u}^0\| +
  \|G(t,u,b)-G(t,u,\tilde{b})\|^1_{T,c} \right).
  \ \Box
  \end{eqnarray*}
 \end{proposition}

 Next Lemma treats  the case in which  a parameter also affects  the linear term of the equation.

\begin{lem}
\label{limeq}
Consider the family of   equations
\begin{equation}\label{eqeps}
    (-i\partial_t+\eps \Delta)u=\eps G(t,u, b)+\phi(t,u, b), \  u(0)=u^0, \ \eps>0,
\end{equation}
depending on the parameters  $\eps>0, \ b \in B $, where $B$ is a compact subset of $\mathbb{R}^N$.
Let $G:[0,T] \times H \times B \rightarrow H$ be continuous in $b$ with respect to each seminorm $\|\cdot\|^1_{T,c}$ and satisfy (\ref{gbip}) with the functions $\beta_c(\cdot)$, the same for all $b \in B$. Let $\phi:[0,T] \times H \times B \rightarrow H$ be continuous.

 Consider the 'limit equation' for $\eps=0$:
 \begin{equation}\label{eq0}
    -i\partial_t u=\phi(t,u,b), \ u|_{t=0}=u^0.
\end{equation}
Assume mild  solutions $\tilde{u}(\cdot,b) \in C([0,T],H)$ of
(\ref{eq0}) to be defined on $[0,T]$ for each $b \in B$.
Then there exists $\eps_0$,  such that for $\eps \in [0, \eps_0)$ the
solutions $u^\eps(t,b)$ of (\ref{eqeps}) exist on $[0,T]$ for each $b \in B$,
and
$$\sup_{t \in [0,T]}\|u^\eps(t,b)-\tilde{u}(t,b)\|_H \rightarrow 0, \ \mbox{as} \ \eps \to 0,   $$
uniformly in $b \in B. \ \Box$
\end{lem}

The proof of the Lemma can be found in  Appendix.

Below we will  formulate and employ  more general continuity
result (Proposition~\ref{codep}) which incorporates perturbations $\phi(t,x)$,  which are fast-oscillating in time, and
with  \eqref{dede},  substituted by a weaker estimate,  based on  {\it relaxation metric}  for the right-hand sides.

\section{Extension of control and controllability}\label{t2}
 In this  Section we introduce {\it extensions of control}  which are  the main tools  for establishing controllability.
 The outcome of the Section  is Proposition~\ref{exst} which establishes  sufficient criterion for controllability in a finite-di\-men\-si\-onal component, wherefrom one  derives in Section~\ref{proapc} criteria for controllability in projections and approximate controllability (Main Result 1).

In what follows
the metrics  ${\mathbf L}^1\left([t_0,t_1],H^s\right),{\mathbf L}^1\left([t_0,t_1],
\mathbb{C}^\kappa\right)$,   $[t_0,t_1] \subset \mathbb{R}$  will be  denoted  both by ${\mathbf L}^1_t$  by abuse of notation.

\subsection{Extensions: approximative lemma}
Consider the NLS equation  (\ref{nls1})-(\ref{sote0}) with controls applied to the modes, indexed by a set $\hat{\mathcal{K}} \subset \mathbb{Z}^d$.

\begin{definition}\label{dfsa} Given a finite set  $\hat{\mathcal{K}} \subset \mathbb{Z}^d$,  we define:

i)  {\it elementary extensions} of $\hat{\mathcal{K}}$, being the sets   $\mathcal{K}=\hat{\mathcal{K}} \bigcup \{2r-s \}$, where $ r,s \in \hat{\mathcal{K}}$ are arbitrary;

 ii)  {\it proper extensions} $\mathcal{K}$  of $\hat{\mathcal{K}}$,  such that  there exist  finite sequences of sets
$\hat{\mathcal{K}}=\mathcal{K}^1 \subset \mathcal{K}^2 \subset \cdots \subset \mathcal{K}^N=\mathcal{K},$
with  $\mathcal{K}^j$ being  elementary extensions of $\mathcal{K}^{j-1}, \ j =2, \ldots $;

 iii)  {\it saturating} sets $\hat{\mathcal{K}}$  such that, each  finite subset $\mathcal{K} \subset \mathbb{Z}^d$ is    proper extension of  $\hat{\mathcal{K}}. \ \Box$
 \end{definition}

 It turns out that saturating sets of modes are essential for controllability in projections and for the approximate controllability.
 Examples of  saturating sets are provided in Section~\ref{sscm}.

The following Lemma  states that controls (energy) fed into  the modes, indexed by  $\hat{\mathcal{K}}$, can be cascaded to  and moreover  can approximately
 control  a larger set $\mathcal{K}$ of modes, whenever $\mathcal{K}$  is  proper extension  of $\hat{\mathcal{K}}$.

\begin{lem}[approximative lemma]
\label{appex}
Let $\mathcal{K}$ be a proper extension of $\hat{\mathcal{K}}$.
Consider  a  family of controls $\{w_k(t,b)| \ k \in \mathcal{K}, \ t \in [0,T]\}$,
parameterized  by $b$ from a compact set $B \subset \mathbb{R}^N$,   and depending  continuously in ${\mathbf L^1_t}$-metric on $b$.
  Then for each $\delta >0$ one can construct a  family
of controls $\{v_k(t,b)| \ k \in \hat{\mathcal{K}}, \ t \in [0,T]\}$
continuously depending on $b$  in ${\mathbf L}^1_t$-metric,
and such that
\begin{equation}\label{evew}
  \| E_T(v(\cdot ,b)) -  E_T(w(\cdot ,b))\| \leq \delta , \ \forall b \in B , .
\end{equation}
for the respective end-point map $E_T$  (see Remark~\ref{pil}) of the NLS equation \eqref{nls1}. $\Box$
\end{lem}

\begin{remark}\label{lodc}
The controls, which appear  in the formulation of the Lemma, correspond to the source terms
     $$ W(t ;b)=\sum_{k \in \mathcal{K}} w_k(t ,b)e^{i k \cdot x}, \
 V(\cdot ,b)=\sum_{k \in \hat{\mathcal{K}}} v_k(t ,b)e^{i k \cdot x},  \ t \in [0,T],  \ b \in B;
$$ 
$V(t,b)$  take  values in the 'low-di\-men\-si\-onal'
space $\mathcal{F}_{\hat{\mathcal{K}}}$ on the contrast to    the 'high-di\-men\-si\-onal' space $\mathcal{F}_{\mathcal{K}}$, which is the range of
$ W(t,b). \ \Box$
\end{remark}

\subsection{Full-di\-men\-si\-onal control}

 Before proving that controllability can be achieved  by means of low-di\-men\-si\-onal controls we formulate a general controllability result for the case, where the control is {\it 'full-di\-men\-si\-onal'}.

\begin{lem}[controllability by 'full-di\-men\-si\-onal' control]
\label{m=1}
Controlled semi-linear equation
\begin{equation}\label{fuco}
-i \partial_t u(t,x) +  \Delta u(t,x) =G(t,u)+\sum_{k \in \hat{\mathcal K}=\mathcal{K}  ^o}w_k(t)e^{ik \cdot x}, \ u(0)=u^0,
\end{equation}
with coinciding sets of controlled  and observed modes $\hat{\mathcal{K}}=\mathcal{K}^o$,
is   time-$T$ solidly controllable in observed $\mathcal{K}^o$-component  for each $T>0$.

In addition, for each $\delta >0$, any bounded subset $S \subset \mathcal{F}_{\mathcal{K}^o}$ is  time-$T$ solidly attainable for the equation (\ref{fuco}) by means of controlled  trajectories $u(\cdot)$, which satisfy  the estimate
$\|(I-\Pi^o)\left(u(T)-u^0\right)\| \leq \delta$, where
$\Pi^o, I-\Pi^o$  are the orthogonal  projections  onto $\mathcal{F}_{\mathcal{K}^o}$ and  its  orthogonal complement. $\Box$
\end{lem}

{\em Proof of Lemma~\ref{m=1}}.
 Take a ball $\mathcal{B}$ in $\mathcal{F}_{\mathcal{K}^o}$. One can assume without lack of generality,   that $\mathcal{B}$ is centered at the origin and  the initial condition
is  $u(0)=0_{H}$.
We will prove that $\mathcal{B}$ is solidly attainable for the controlled equation (\ref{fuco}).

Restrict (\ref{fuco})  to  a small interval $[0,\eps]$ to be specified later.
Proceed with the time substitution $t=\eps \tau , \ \tau \in
[0,1]$ under which (\ref{fuco})  takes form:
\begin{equation}\label{resc}
- i \partial_\tau u +  \eps \Delta u =\eps G(t,u)+\eps \sum_{k \in \mathcal{K}^o} w_k(t)e^{ik \cdot x}, \ u(0)=0,
  \ \tau \in [0,1].
\end{equation}

Fix $\gamma >1$.
For each $b \in B_\gamma = \gamma \mathcal{B},$ consider the control
$$w(\cdot;b)=-i\eps^{-1}\sum_{k \in \mathcal{K}^o} b_k e^{ik \cdot x},$$
which we substitute into  (\ref{resc}) getting
$$-i\partial_\tau u +  \eps \Delta u =\eps G(t,u)-i\sum_{k \in \mathcal{K}^o} b_k e^{ik \cdot x}, \ u(0)=0,
  \ \xi \in [0,1].
$$

For  $\eps =0$ we get  the 'limit equation'
\begin{equation}\label{0resc}
 \partial_\tau u = \sum_{k \in \mathcal{K}^o} b_k e^{ik \cdot x}, \ u(0)=0,
  \  , \ \tau \in [0,1].
\end{equation}
Let $E_1^0$  be the  time-$1$
end-point map of (\ref{0resc}).  In the basis $e^{ik \cdot x}$ the map $b \mapsto \Phi^0(b)= E_1^0(w(t;b))$  has form
$\{b_k| \ k \in \mathcal{K}^o\} \mapsto u(1)=\sum_{k \in \mathcal{K}^o} b_k e^{ik \cdot x} $.

 Obviously the map $b \mapsto \Phi^0(b)=\Pi^o \circ E_1^0 (w(t;b))$  coincides   with  the
identity map $\mbox{Id}_{B_\gamma}$ and $(I-\Pi^o)E_1^0 (w(t;b))=0$.

According to the Lemmae~\ref{wecopar},\ref{limeq}  the maps
$\Phi^\eps: b \mapsto E^\eps_1 (w(\cdot ,b))$,  with $E^\eps_1$ being the end-point maps  of the control systems  (\ref{resc}), are continuous and
  $\|\Phi^\eps - \Phi^0\|_{C^0(B_\gamma)} \rightarrow 0$  as $\eps \rightarrow 0$.

By the degree theory
argument there exists $\eps_0$ such that $\forall \eps \leq
\eps_0$ the image of $\left(\Pi^0 \circ \Phi^\eps \right)(B_\gamma)$ covers $\mathcal{B}$
solidly.

We proved  controllability in the observed component for small times $T>0$. Controllability in any time $T>0$
can be concluded by applying zero control on a time interval  $[0,T-\delta]$ (the trajectory is maintained in a bounded domain due to the conservation law) and then employing the previous reasoning on the interval $[T-\delta ,T].  \ \Box$

\begin{remark}\label{smow}
Without lack of generality we may assume, that the controls $w(t ,b)$, constructed in the Lemma  are smooth with respect to $t$ and that any finite number of derivatives   $\frac{\partial^j w}{\partial t^j}(\cdot , b)$ depend continuously  on $b \in B$ in ${\mathbf L}^1_t$-metric. Indeed smoothing $w(t,b)$ by the  convolution with a smooth $\eps$-approximation  of Dirac function, one gets a  family of smooth controls $w^\eps(t ,b)$, which   provides solid controllability, for small $\eps>0$. The continuous dependence of
$\frac{\partial^j w^\eps}{\partial t^j}(\cdot , b)$ on $b$ in ${\mathbf L}^1_t$-metric  is verified directly. $\Box$
\end{remark}

\subsection{Controllability in finite-di\-men\-si\-onal component via extensions}
\label{53}

The following result  is a corollary of Lemmae~\ref{appex},\ref{m=1}.

\begin{proposition}[controllability in  observed component]
\label{exst}
If the  set of observed modes $\mathcal{K}^{o}$ is   proper extension of the set  of controlled modes
$\hat{\mathcal{K}}$,
then the NLS equation \eqref{nls1}-\eqref{sote0}
is solidly controllable in the observed $\mathcal{K}^{o}$-component.   $\Box$
\end{proposition}

\begin{proof}
Let $S$ be a compact subset of $\mathcal{F}_{\mathcal{K}^o}=\mbox{span}\{e_k| \ k \in \mathcal{K}^o\}$. According to the  Lemma~\ref{m=1} we can choose a  family of $\mathcal{F}_{\mathcal{K}^o}$-valued controlling source terms  $W(\cdot , b)$ by which $S$ is solidly attainable.
If  a family $V(t;b)$ in \eqref{sote0} satisfies the conclusion of the Approximative Lemma~\ref{appex} and $\delta>0$ is small enough, then
$\Pi^o \circ E_T(V(t;b))$ covers $S$ solidly.
\end{proof}

\begin{corollary}\label{safdc}
If the set of  controlled modes $\mathcal{K}$ is saturating, then the NLS equation is solidly controllable in each finite-di\-men\-si\-onal component.
$\square$
\end{corollary}

Examples of saturated sets are provided in Section~\ref{sscm}.


\subsection{Proof of Approximative Lemma~\ref{appex}}
\label{ale}

It suffices to prove the Approximative Lemma  for  $\mathcal{K}$ being  an {\rm elementary extension} of  $\hat{\mathcal{K}}$, the rest being accomplished by induction. Let $\mathcal{K}=\hat{\mathcal{K}} \bigcup \{2r-s\},\ r,s \in \hat{\mathcal{K}}. $

It is convenient to proceed  with the time-variant change of basis in $H$,  passing from the exponentials
$e^{ik \cdot x}$ to the exponentials
$$f_k=e^{i(k \cdot x+|k|^2t)}, \ k \in \mathbb{Z}^d.$$
Note that
 $(-i\partial_t +\Delta)f_k=0, \ \forall k \in \mathbb{Z}^d$.

We take a  $\mathcal{F}_{\mathcal{K}}$-valued family of the controlling source terms
\begin{equation}\label{wk'}
    b \mapsto W(t , b)=\sum_{k \in \mathcal{K}}w_k(t ; b)f_k,
\end{equation}
parameterized by $b$ from a compact $B \subset \mathbb{R}^N$, and  wish to construct a family of the controlling source terms $V(t;b)=\sum_{k \in \hat{\mathcal{K}}}v_k(t ; b)f_k$, which  satisfy (\ref{evew}) and whose range $\mathcal{F}_{\hat{\mathcal{K}}}$ has  one dimension less .

\subsubsection{Time-variant substitution}

We will seek  the family $b \mapsto V (t ,b)$ in   the form
\begin{equation}\label{twoc}
V(t ,b)=\tilde{V}(t , b)+  \partial_t v_r(t , b)f_r+ \partial_t v_s(t , b)f_s,
\end{equation}
where   $\tilde{V}(t , b)=\sum_{k \in \hat{\mathcal{K}}}\tilde{v}_k(t ; b)f_k$,
and  the Lipschitzian functions $v_r(t ,b), v_s(t ,b)$
will be  specified  in the course of the proof.
For some time we will omit dependence on $b$ in the notation.

Feeding the controls (\ref{twoc}) into the right-hand side of  equation (\ref{nls1})  we get
\begin{equation}\label{fein}
(-i\partial_t +\Delta)u=|u|^2u+\tilde{V}(t)+\dot{v}_r(t)f_r+\dot{v}_s(t)f_s.
\end{equation}

This equation can be given form
\begin{equation}\label{uiv}
(-i\partial_t +\Delta)\left(u-iV_{rs}(t)\right)=|u|^2u+\tilde{V}(t),
\end{equation}
where $V_{rs}(t)=v_r(t)f_r+ v_s(t)f_s$.

By a time-variant substitution
$$ u^*=u-iV_{rs}(t),$$
we  transform (\ref{uiv}) into  the equation:
\begin{eqnarray}\label{uste}
(-i\partial_t +\Delta)u^*=|u^*+iV_{rs}(t)|^2(u^*+iV_{rs}(t))+\tilde{V}(t)
= \\=|u^*|^2u^*  -i(u^*)^2\bar{V_{rs}}+ 2i|u^*|^2V_{rs}-V_{rs}^2\bar{u^*}+2u^*|V_{rs}|^2+i|V_{rs}|^2V_{rs}+\tilde{V}(t).
\nonumber
\end{eqnarray}

Imposing the constraints
\begin{equation}\label{vt0}
 v_r(0)=v_s(0)=0, \ v_r(T)=v_s(T)=0,
\end{equation}
 we achieve:
$u(0)=u^*(0)$, $u(T)=u^*(T)$, and hence the end-point maps $E_T,E^*_T$ of  the controlled equations (\ref{fein}) and (\ref{uste}) coincide provided (\ref{vt0}) hold.

\subsubsection{Fast oscillations and resonances}
 Now  we put {\it fast-oscil\-lations,} into the game, choosing $V_{rs}(t)$ in (\ref{uste}) of   the form
\begin{equation}\label{vrexp}
    V_{rs}(t)=v_r(t)f_r+ v_s(t)f_s=e^{i(t/\eps + \rho(t))}\check{v}_r(t)f_r+e^{i2t/\varepsilon}\check{v}_s(t)f_s,
    \end{equation}
    where $\check{v}_r(t),\check{v}_s(t),\rho(t)$ are Lipschitzian {\it real-valued} functions, which together with   small $\varepsilon >0$,   will be specified in the course of the proof.

We  classify those terms at the right-hand side of (\ref{uste}),  which contain $V_{rs}, \bar{V_{rs}}$,  as
{\it non-resonant} and {\it resonant} with respect to the substitution (\ref{vrexp}).

We call  a  term non-resonant if, after  the substitution \eqref{vrexp} the term results in a sum of fast-oscillating factors of the form  $p(u^*,V_{rs},t)e^{i \beta t/\varepsilon}, \ \beta \neq 0$, where $p(u,V_{rs},t)$ is polynomial in $u^*,\bar{u^*},V_{rs},\bar{V_{rs}}$, with coefficients, which are  Lipschitzian in  $t$ and independent of  $\varepsilon$. Otherwise, i.e. $\beta=0$, the term is resonant.

Crucial fact, to be established below, is that the {\it influence of non-resonant (fast-oscillating) terms onto the end-point map can be made arbitrarily small}, if the frequency $\beta/\varepsilon $ of the oscillating factor $e^{i \beta t/\varepsilon}$  is chosen sufficiently large.

 Direct verification shows that the terms
 $$i(u^*)^2\bar{V}_{rs}, \ 2i|u^*|^2V_{rs}, \ V_{rs}^2\bar{u}^*$$
 at the right-hand side of (\ref{uste}) are all non-resonant with respect to (\ref{vrexp}).

 \subsubsection{Resonant monomials in the quadratic  term $2u^*|V_{rs}|^2$: an obstruction}

Consider the quadratic term $2u^*|V_{rs}|^2$, which after  the substitution (\ref{vrexp}) takes form
 $$2u^*|V_{rs}|^2=2u^*\left(|\check{v}_r(t)|^2+|\check{v}_s(t)|^2\right)+
 4u^*\check{v}_r(t)\check{v}_s(t)\mbox{Re}\left(e^{-it/\eps}e^{i\rho(t)}
 f_r\bar{f}_s \right).   $$
 The last addend in the parenthesis is non-resonant, while
  the resonant  term $2u^*(|\check{v}_r(t)|^2+|\check{v}_s(t)|^2)$ is an example of so-called  {\it obstruction to controllability} in the
  terminology of geometric control.

      We can not annihilate or compensate this term but, as far as the group $e^{it\Delta}$ corresponding to the {\it linear}  Schroedinger equation is quasiperiodic,
             one can  impose conditions on the functions $\check{v}_r(\cdot),\check{v}_s(\cdot)$  in such a way, that for a chosen $T>0$ the influence of the obstructing term  onto the end-point map $E_T$ will be nullified.

Indeed,  proceeding  with   the time-variant substitution:
\begin{equation}\label{34}
u^\star=u^*e^{-2i \Upsilon(t)}, \ \Upsilon(t)=\int_0^t(|\check{v}_r(t)|^2+|\check{v}_s(t)|^2)d\tau ,
\end{equation}
one gets  for $u^\star$ the equality:
$$
(-i\partial_t +\Delta)u^\star e^{2i \Upsilon(t)}=
  (-i\partial_t +\Delta)u^*-2u^*(|\check{v}_r(t)|^2+|\check{v}_s(t)|^2).
$$
The equation (\ref{uste}) rewritten for $u^\star$ becomes
\begin{eqnarray}\label{eust}
(-i\partial_t +\Delta)u^\star=|u^\star|^2u^\star -i(u^\star)^2\bar{V_{rs}}e^{2i \Upsilon(t)}+2i|u^\star|^2V_{rs}e^{-2i \Upsilon(t)}- \nonumber  \\
-V_{rs}^2\bar{u^\star}e^{-4i \Upsilon(t)}+  4u^*2 \mbox{Re}\left(e^{i(t/\eps + \rho(t))}v_r(t)\bar{v}_s(t)\right)e^{-2i \Upsilon(t)}+ \\    + \tilde{V}(t)e^{-2i \Upsilon(t)} +i|V_{rs}|^2V_{rs}e^{-2i \Upsilon(t)}. \nonumber
\end{eqnarray}

    For the sake of maintaining the  end-point map $E_T$ unchanged, one imposes  {\it isoperimetric
conditions} on $\check{v}_r(t),\check{v}_s(t)$
\begin{equation}\label{ivrvs}
  \int_0^T (|\check{v}_r(t)|^2+|\check{v}_s(t)|^2)dt=\Upsilon(T)=\pi N,  \ N \in \mathbb{Z};
\end{equation}
 this
guarantees for the respective trajectories   $u^\star(0)=u^*\mbox(0),\ u^\star(T)=u^*\mbox(T)$.

\begin{remark}\label{upsi}
Although  the right-hand side of (\ref{eust}) has gained the 'oscillating factors' of the form  $e^{-i \mu \Upsilon(t)}$, the notions of resonant and non-resonant terms will not suffer changes, as long as $e^{-i\mu  \Upsilon(t)}$ is not 'fast oscillating'; for this sake in  the ongoing  construction $\Upsilon(t)$ will be chosen  bounded uniformly in $t,b$ and  $\eps>0. \ \Box$
\end{remark}

We introduce the notation $\tilde{\mathcal{N}}^\eps(u,t)$ for the sum of the non-resonant terms at the right-hand side of (\ref{eust}) and arrive  to the equation
\begin{equation}\label{exs}
 (-i \partial_t +  \Delta) u^\star   =|u^\star|^2u^\star  +\tilde{V}(t)e^{-2i \Upsilon(t)}+i|V_{rs}|^2V_{rs}e^{-2i \Upsilon(t)}+\tilde{\mathcal{N}}^\eps(u,t).
  \end{equation}

\subsubsection{Extending the control via cubic resonance monomial}
The only resonant monomial in the cubic term
 $i|V_{rs}|^2V_{rs}e^{-2i \Upsilon(t)}= i(V_{rs})^2\bar{V}_{rs}e^{-2i \Upsilon(t)},$
with $V_{rs},\Upsilon$,  defined by  \eqref{vrexp}, \eqref{34},
is
\begin{equation}\label{eros}
e^{2i(\rho (t)-\Upsilon(t))}\check{v}^2_r(t)\check{v}_s(t)f_r^2\bar{f_s}.
\end{equation}
We  join all the non-resonant monomials
of this term to $\tilde{\mathcal{N}}^\eps(u,t)$ in (\ref{exs}).

Recalling  that $f_m=e^{i(m\cdot x +|m|^2)t}$, we compute
$$f_r^2\bar{f_s}= e^{i((2r-s)\cdot x +(2|r|^2-|s|^2)t)}=f_{2r-s}e^{i((2|r|^2-|s|^2-|2r-s|^2)t)}=f_{2r-s}e^{-2i|r-s|^2t}, $$
and rewrite (\ref{eros}) in the form
$$\check{v}^2_r(t)\check{v}_s(t)e^{2i(\rho(t)-|r-s|^2t- \Upsilon(t))}f_{2r-s},$$
   seeing it  as an   {\it extending control} for the mode $f_{2r-s}$.

The equation (\ref{exs}) becomes now
\begin{eqnarray}\label{excu}
 (-i \partial_t +  \Delta) u^\star   =|u^\star|^2u^\star  + \tilde{V}(t)e^{-2i \Upsilon(t)}+ \nonumber \\
 +\check{v}^2_r(t)\check{v}_s(t)e^{2i(\rho(t)-|r-s|^2t- \Upsilon(t))}f_{2r-s}+\tilde{\mathcal{N}}^\eps(u,t).
  \end{eqnarray}

Now we  take care of the addend
\begin{equation}\label{vet}
  \tilde{V}(t)e^{-2i \Upsilon(t)}+\check{v}^2_r(t)\check{v}_s(t)e^{2i(\rho(t)-|r-s|^2t- \Upsilon(t))}f_{2r-s}
\end{equation}
 at the right-hand side of (\ref{excu}).
We wish to choose families of functions $\tilde{V}(t;b),\check{v}_r(t;b), \check{v}_s(t;b)$ in such a way that \eqref{vet}
 approximates $W(t;b)$ (see (\ref{wk'})) in ${\mathbf L}^1_t$-metric uniformly in $b \in B$.

Let
\begin{equation}\label{vtb2}
\tilde{V}(t;b)=\hat{W}(t;b)e^{2i\Upsilon (t;b)},
\end{equation}
where
$\hat{W}(t;b)= \sum_{k \in \hat{\mathcal{K}}, k \neq 2r-s}w_k(t ; b)f_k ,$
is obtained by omission of  the summand $w_{2r-s}f_{2r-s}$ in  $W(t;b)$.

The controls $\check{v}_r(t;b), \check{v}_s(t;b)$ will be constructed according to  the
following
\begin{lem}\label{5.4}
For a family of controls $b \mapsto w(t;b) \in {\mathbf L}^\infty [0,T]$,  continuous in ${\mathbf L^1_t}$-metric,  and  any  $\eps' >0$
one can construct families of real-valued functions
\begin{equation}\label{vch}
b \mapsto  \check{v}_r(t;b,\eps'),b \mapsto \check{v}_s(t;b,\eps'),
\end{equation}
such that:
i) the functions \eqref{vch} are  Lipschitzian in $t$;
$$ii)  \ \hat{v}_r(0)=\hat{v}_s(0)=0, \ \hat{v}_r(T)=\hat{v}_s(T)=0,$$
iii) the functions \eqref{vch} and their partial derivatives in $t$  depend on $b$ continuously in ${\mathbf L^1_t}$-metric; iii) for each $b, \eps'$ the conditions (\ref{vt0}),(\ref{ivrvs}) hold for them; iv) the ${\mathbf L^2_t}$-norms of the functions \eqref{vch} are equibounded for all $\eps'>0, b \in B$; and  v)
for
$$
D^{\eps'}_{rs}= \check{v}^2_r(t;b,\eps')\check{v}_s(t;b,\eps')- |w_{2r-s}(t,b)|,
$$
the estimate
\begin{equation}\label{aper}
  \|D^{\eps'}_{rs}\|_{{\mathbf L^1_t}}=\int_0^T \left|\check{v}^2_r(t;b,\eps')\check{v}_s(t;b,\eps')- |w_{2r-s}(t,b)|\right|dt \leq \eps' .
\end{equation}
holds uniformly in $b \in B. \ \Box$
\end{lem}

Lemma \ref{5.4} is proved in the Appendix. Meanwhile  we construct   the family
$\rho(t;b)$.

\begin{lem}\label{wappr}
Given  the family (\ref{vch}), constructed in the  previous Lemma,
there exists a family of Lipschitzian functions $\rho(\cdot ;b)$, such that $b \mapsto \rho(\cdot ;b)$ and $b \mapsto \partial_t\rho(\cdot ;b)$  are  continuous in ${\mathbf L^1_t}$-metric and
\begin{equation}\label{v)}
\int_0^T\left|\check{v}^2_r(t;b,\eps')\check{v}_s(t;b,\eps')e^{2i(\rho(t)-|r-s|^2t- \Upsilon(t))}-w_{2r-s}(t,b)\right|dt \leq \eps'. \ \Box
\end{equation}
\end{lem}

\begin{proof} We  choose
\begin{equation}\label{38}
\rho(t;b)=\frac{1}{2}Arg\left(w_{2r-s}(t,b)\right)+|r-s|^2t+\Upsilon(t;b).
\end{equation}
  As in the Remark~\ref{smow} we may think that $w_{2r-s}(t,b)$ are smooth in $t$ and hence $\rho(t;b)$ are  Lipschitzian in $t$. The dependence of $\rho$ and $\partial \rho /\partial t$ on $b$ is continuous in ${\mathbf L}^1_t$-metric.
By  (\ref{aper}),(\ref{38}) we conclude (\ref{v)}).
\end{proof}

\begin{remark}
         For each $\eps , \eps ' >0$ the functions $v^r(t,b),v^s(t,b)$ defined by   \eqref{twoc}, \eqref{vtb2}, \eqref{vrexp}, as well as their derivatives in $t$, depend continuously on $b$  in
$\mathbf{L^1_t}$-metric.
\end{remark}

By the construction  the  map $b \mapsto E_T(v(t,b,\eps ,\eps '))$, where $E_T$ is the end-point map of the equation \eqref{nlsrs},
coincides with the end-point map $b \mapsto  E^{\eps,\eps'}_T(b)$ of the equation
\begin{equation}\label{exw}
 (-i \partial_t +  \Delta) u^\star
 =|u^\star|^2u^\star  +W(t,b)+ D^{\eps
 '}_{rs}(t) +\tilde{\mathcal{N}}^\eps(u^\star,t,b), \  u^\star(0)=u^0.
 \end{equation}

The proof of Approximative Lemma~\ref{appex} would be completed by the following

\begin{lem}
\label{tecrel}
The end-point map $E^{\eps \eps '}_T (b)$ of the system (\ref{exw}) calculated for the family of controls,
 defined by the Lemmae~\ref{5.4},\ref{wappr} converges to the end-point map $E^{lim}_T$ of the 'limit system'
 \begin{equation}\label{exl}
 (-i \partial_t +  \Delta) u^\star
 =|u^\star|^2u^\star  +W(t,b),
 \end{equation}
  uniformly in $b$  as $\eps+\eps '  \rightarrow 0. \ \Box$
\end{lem}

Would  the term  $\tilde{\mathcal{N}}^\eps(u^\star,t,b)$ be missing in (\ref{exw})  we could derive the conclusion of the Lemma from the Proposition~\ref{wecopar}.
In the presence  of the fast-oscillating term $\tilde{\mathcal{N}}^\eps(t,u)$ the needed passage to a limit system  will be established by virtue of  Proposition~\ref{codep} of the next Subsection.


\subsection{On continuity of solutions in the  right-hand side with respect to the relaxation metric}
\label{corel}

The results presented   in this Subsection, regard continuous dependence  of  the mild solutions of NLS equation with respect to  the perturbations of its right-hand side, which are small in so-called {\it relaxation seminorms}.

 The seminorms are suitable for treating fast oscillating terms.  In finite-di\-men\-si\-onal context the respective continuity   results are part of  {\it theory of relaxed controls}. Several relaxation results for semilinear systems in Banach spaces can be found in  \cite{Fat,Fr}.
Below we provide  version adapted for  the  proof of Lemma~\ref{tecrel}.

Let us come back to  the semilinear equation (\ref{nlsmod})
and its  perturbation (\ref{nlsper}).
We assume the perturbations $\phi:[0,T] \times H \to H$ in (\ref{nlsper}) to belong to a family $\Phi$,
which satisfies the following conditions:
\begin{enumerate}
  \item[(i)] elements of  $\Phi$  are continuous functions;
  \item[(ii)]  the family $\Phi$ is equibounded and equi-Lipschitzian, which  means that each $\phi \in \Phi$ together with $G:[0,T] \times H \rightarrow H$   satisfy the properties
\eqref{congfi},(\ref{gbip}), (\ref{fbip}) with  the same functions $\beta_c(t)$.
  \item[(iii)] the set
$\{\phi(t,u(t))| \ t \in [0,T], \phi \in \Phi\} $  is completely bounded in $H$ for each  $u(\cdot) \in C([0,T],H)$.\footnote{The property  of complete boundedness would follow, for example, from the  {\it complete boundedness  of the sets} $\Phi(t,u)=\{\phi(t,u)|\  \phi \in \Phi\} $ for each {\it fixed} couple $(t,u)$ together with  {\it upper semicontinuity of the set-valued map} $(t,u) \mapsto \Phi(t,u)$.}
\end{enumerate}

 We introduce the
{\it relaxation seminorm
$\|\cdot\|^\mathrm{rx}_{c}$}  for the elements of $\Phi$ by the formula:
  $$\|\phi\|^\mathrm{rx}_{c}=\sup_{t,t' \in[0,T],\|u\|\leq c}\left\|
\int\limits_t^{t'} \phi(\tau,u)d \tau \right\|_H.$$
            As one can see the relaxation seminorms of fast-oscillating functions are small. For example $\|f(t)e^{it/\eps}\|_c^{rx} \rightarrow 0$, as $\eps \rightarrow 0$ for each function $f \in {\mathbf L}^1[0,T]$ (and each $c$) accoding to Lebesgue-Riemann lemma.

Now we formulate the needed continuity result from which Lemma~\ref{tecrel} will follow.

 \begin{proposition}
 \label{codep}
Let a mild solution $\tilde{u}(t) \in C([0,T],H)$ of the NLS equation (\ref{nlsmod})   satisfy  $\sup_{t \in [0,T]}\|u(t)\|_H <c$.
Let the family $\Phi$ of perturbations satisfy the conditions (i)-(iii) just introduced.  Then $\forall \eps >0 \exists \delta >0$
such that,   whenever $\phi \in \Phi, \ \|\phi\|^\mathrm{rx}_{c} +\|u^0-\tilde{u}^0\|_H < \delta $ ,
then the mild solution $u(t)$ of the perturbed equation (\ref{nlsper})
 exists on the interval $[0,T]$, is unique and satisfies the estimate
$\sup_{t \in [0,T]}\left\|u(t)-\tilde{u}(t)\right\|_H <\eps . \ \Box$
\end{proposition}

{\it Sketch of the proof of the Proposition~\ref{codep}.} Under the assumptions of the Proposition the solution of the equation (\ref{nlsper}) exists locally and is unique (see  Proposition~\ref{loex}).
Global existence will follow from the boundedness  of the $H$-norm of the solution on $[0,T]$.

We start with  the estimate (\ref{gro54})  by which:
    $$\|u(t)-\tilde{u}(t)\| \leq  \left(\|u^0-\tilde{u}^0\|+\left\|\int_0^t e^{-i\tau \Delta}\phi(\tau,\tilde{u}(\tau))d\tau \right\|\right)C'e^{C\int_0^t \beta_c (\tau)d\tau}.$$

The conclusion of the Proposition~\ref{codep} would follow  from

\begin{lem}\label{coconv}Let the family $\Phi$ satisfy the assumptions of the Proposition~\ref{codep}, and let $\tilde{u}(t)$ be  solution of (\ref{nlsmod}). Then $\forall \eps >0, \ \exists \delta>0$ such that $\forall \phi \in \Phi$:
  $$        \|\phi\|^{rx}< \delta \Rightarrow   \left\|\int_0^t e^{-i\tau \Delta}\phi(\tau,\tilde{u}(\tau))d\tau \right\| < \eps . \ \Box $$
\end{lem}

Proof of the Lemma can be found in  Appendix.

Let us  remark on validity of conditions of the Proposition~\ref{codep} for the limit equation \eqref{exl} and for its perturbation \eqref{exw}.

The   perturbation  $D^{\eps '}_{rs}(t) +\tilde{\mathcal{N}}^\eps(u^\star,t,b)$ at the right-hand side of (\ref{exw}) can be seen as an operator $\phi : [0,T] \times H \to H$:
\begin{eqnarray*}
\forall u(\cdot) \in H: \  \phi^\eps: u(x) \mapsto \\ W^{0}(t,x)+   u(x)W^{11}(t,x)+\bar{u}W^{12}(t,x)+u^2(x)W^{21}(t,x)+ |u(x)|^2W^{22}(t,x),
\end{eqnarray*}
where $W^{ij}(t,x)=w(t)e^{ik \cdot x}e^{i \rho(t)}e^{iat/\eps}$, and  $w(t),\rho(t)$ are Lipschitzian.
The continuity, equiboundedness and equi-Lipschitzianness are concluded by application of 'Product Lemma'~\eqref{prol}.

To  confirm the complete boundedness assumption for $\phi^\eps(t,\tilde{u}(t))$, where $\tilde{u}(t)$ is continuous on $[0,T]$, we substitute the factors $e^{iat/\eps}$ in the coefficients $W_{ij}(t,x)$ by $e^{i \theta}, \ \theta \in \mathbb{T}^1$  and arrive to a function $\psi(t,\theta), \ \theta \in \mathbb{T}^1$, which is  continuous on $[0,T] \times \mathbb{T}^1$, and whose compact range contains the range of $\phi^\eps(t,\tilde{u}(t))$ for all $\eps>0$.

\section{Controllability proofs (Main result 1)}
\label{proapc}

We have established above (Co\-rol\-la\-ry~\ref{safdc}) that whenever    set $\hat{\mathcal{K}}$ of the controlled modes is saturating, then the NLS equation
 \eqref{nls1}-\eqref{sote0} is solidly controllable in projection on any 'coordinate' subspace  $\mathcal{F}_{\mathcal{K}}$ with finite $\mathcal{K} \subset \mathbb{Z}^d$.

In the Section we prove that  this  implies 
 controllability in projection on any
finite-di\-men\-si\-onal  subspace of $H^s, s >d/2$ and $H^s$-approximate controllability. 

Together with the  description of the  classes of saturating sets, provided in Section~\ref{sscm}, this would complete the proof of the Main result~1.

\subsection{Approximate controllability}

Let us fix $\tilde{\varphi}, \hat{\varphi} \in H^s$ and $\eps >0$,
we wish to steer the NLS equation  from
$u^0=\tilde{\varphi}$ to the $\eps$-neighborhood of $\hat{\varphi}$
in the $H^s$-metric.

Consider the Fourier expansions  for $\tilde{\varphi}, \hat{\varphi}$ with respect to  $e^{i k\cdot x},\ k \in \mathbb{Z}^d$.
 Denote by $\Pi_N$ the orthogonal projection of $\varphi$ onto the
space of modes $\mathcal{F}_N=\mbox{Span}\{e^{i k\cdot x}, \ |k| \leq \}$. Obviously  $\Pi_N(\tilde{\varphi})\rightarrow
\tilde{\varphi},\Pi_N(\hat{\varphi}) \rightarrow \hat{\varphi}$ in
$H^s$ as $N \rightarrow \infty$.

Choose such $N$  that the $H^s$-norms  of
$(I-\Pi_N)\tilde{\varphi}$ and $(I-\Pi_N)\hat{\varphi}$ are $\leq \eps/4$.

By Lemma~\ref{m=1}
 there exists a family of controlling source terms
$W(b)= \sum_{|k| \leq N}w_k(t;b)f_k$
such that $\Pi_N( W(b))$ covers
$\Pi_N (\hat{\varphi})$ solidly and in addition
  $\|(I-\Pi_N )(E_T( W(b))-\tilde{\varphi})\| \leq \eps/4$. Then   $\|(I-\Pi_N) E_T( W(b))\| \leq \eps/2$.

If  a set  $\hat{\mathcal K}$ of controlled modes is saturating, then
$\{k|\ |k| \leq N\}$ is a proper extension  of $\hat{\mathcal K}$.
  By Approximative Lemma~\ref{appex} there exists family of controlling source terms
  $V(b)=\sum_{k \in \hat{\mathcal K}}v_k(t;b)f_k$ such that
      $$\| E_T(V(b)) -  E_T(W(b))\| \leq \eps/4 , \ \forall b \in B , $$
      and $ \Pi_N E_T(V(b))$ covers the point $\Pi_N (\hat{\varphi})$.
 Then  $\forall b: \ \|(I-\Pi_N) E_T (V(b))\| \leq 3\eps/4$ and for some $\hat{b}$:
  $\Pi_N E_T(V(\hat{b}))=\Pi_N \hat{\varphi}$. Then
  $\|E_T(V(\hat{b}))-\hat{\varphi}\| \leq \eps. \ \Box$

\subsection{Controllability in finite-di\-men\-si\-onal projections}

Let $\mathcal L$ be  $\ell$-dimen\-si\-onal subspace of $H^s$ and
$\Pi^{\mathcal L}$ be the orthogonal projection of $H^s$ onto
${\mathcal L}$.

First we  construct a finite-di\-men\-si\-onal
{\it coordinate} subspace which is projected by $\Pi^{\mathcal L}$ onto
${\mathcal L}$. Moreover for each $\eps >0$
one can find  a finite-di\-men\-si\-onal {\it coordinate subspace} $\mathcal{L}^{\mathcal{C}}$  and   its  $\ell$-di\-men\-si\-onal (non-coordinate) subsubspace
${\mathcal L}_\eps$, which is $\eps$-close to ${\mathcal L}$. The
latter means that not only $\Pi^{\mathcal L}{\mathcal
L}_\eps={\mathcal L}$ but also the isomorphism $\Pi_\eps=\Pi^{\mathcal L}|_{{\mathcal
L}_\eps}$ is $\eps$-close to the identity operator.
It is an easy linear-algebraic construction; which can be found in  \cite[Section 7]{AS06}.

Without lack of generality we may assume that the orthogonal projection $\Pi_{\mathcal C}$ onto $\mathcal{L}^{\mathcal{C}}$ satisfies:
$\|\Pi_{\mathcal C}(\tilde{\varphi})-\tilde{\varphi}\|_{H^s} \leq \eps$.

As far as the set $\hat{\mathcal K}$ of controlled modes is saturating, $\mathcal{C}$
is proper extension of $\hat{\mathcal K}$ and the system is solidly
controllable in the observed component $q^{\mathcal{C}}$.

Let $\mathcal{B}$ be a ball in  $\mathcal{L}$ centered at the origin. Consider $\mathcal{B}^\eps = (\Pi_\eps)^{-1}\mathcal{B}$; obviously $ \mathcal{B}^\eps \subset \mathcal{L}^\eps \subset \mathcal{L}^{\mathcal{C}}$. We take a ball $\mathcal{B}_\mathcal{C}$ in $\mathcal{L}^{\mathcal{C}}$, which contains $ \mathcal{B}^\eps$ and hence $\Pi_{\mathcal{L}}(B_\mathcal{C}) \supset B$.

Reasoning as in the previous Subsection one  establishes existence
of a family of  controls
  $V(b)=\sum_{k \in \hat{\mathcal K}}v_k(t;b)f_k$ such that
     $ \Pi_{\mathcal{C}}E_T(V(b))$ covers $\mathcal{B}_{\mathcal{C}}$ solidly and
$\forall b: \ \|(I-\Pi_{\mathcal{C}}) E_T (V(b))\| \leq 2 \eps .$

Then choosing $\eps>0$ sufficiently small we achieve that
$$\Pi^{\mathcal{L}}E_T(V(b))= \Pi^{\mathcal{L}}\left(\Pi_{\mathcal{C}}+ (I-\Pi_{\mathcal{C}})\right)E_T(\hat{V}(b))$$
covers $\mathcal{B}. \ \Box$

\section{Lack of exact controllability proof (Main result 2)}
Let us represent the controlled cubic defocusing NLS equation (\ref{nls1})-(\ref{sote0}) in the form
\begin{equation}\label{nlsdot}
    (-i\partial_t+\Delta)u=|u|^2u+\sum_{k \in \hat{\mathcal K}}\dot{w}_k(t) f_k, \ u|_{t=0}=u^0 \in H^s,
\end{equation}
where  $f_k=e^{i(k \cdot x +|k|^2t)}$ and $\hat{\mathcal K} \subset \mathbb{Z}^d$ is a finite set.
The controls $\dot{w}_k(t) \in {\mathbf L}^{1}\left([0,T],\mathbb{C}\right)$  are
time-derivatives of absolutely continuous functions
$w_k(t)$,  $w_k(0)=0$.
    In this Section ${\mathbf W}^{1,1}([0,T],\mathcal{F}_{\hat{\mathcal K}})$ stays for the  space of $\mathcal{F}_{\hat{\mathcal K}})$-valued  absolutely continuous functions, vanishing at $t=0$.

 Global existence and uniqueness results for solution of this equation in  $C([0,T],H^s)$  is classical (Section~\ref{prel}).

Consider the end-point map $E_T: \left(\sum_{k \in \hat{\mathcal K}}\dot{w}_k(t) f_k\right) \mapsto u|_{t=T}$ which maps the  space of
  controlling source terms from ${\mathbf L}^1([0,T],\mathcal{F}_{\hat{\mathcal K}})$ into the state space $H^s$. The image of $E_T$  is  time-$T$ attainable set $\mathcal{A}_{T,u^0}$ of the controlled equation (\ref{nlsdot}).

Introducing
$W(t)=\sum_{k \in \hat{\mathcal K}}w_k(t) f_k \in {\mathbf W}^{1,1}([0,T],\mathcal{F}_{\hat{\mathcal K}})$ we bring (as in  Subsection~\ref{ale}) the equation (\ref{nlsdot}) to  the form
$(-i\partial_t+\Delta)(u-iW(t))=|u|^2u, $ and after another time-variant substitution
$u-iW(t,x)=u^*(t)$
 to  the form
\begin{equation}\label{red}
  (-i\partial_t+\Delta)u^*=|u^*+iW(t)|^2(u^*+iW(t)), \  u|_{t=0}=u^0,
\end{equation}
which we look at as  a semilinear control system with the {\it inputs} $W(\cdot)$ belonging to ${\mathbf W}^{1,1}([0,T],\mathcal{F}_{\hat{\mathcal K}})$.
Obviously for each  input $W(\cdot)$ the
solution of (\ref{red}) exists and is unique on $[0,T]$.

Introduce the input-trajectory map  $E^*: W(\cdot) \mapsto u^*(\cdot)$ of the equation (\ref{red}).
The following result is essentially a corollary of Proposition~\ref{wecodep}.

\begin{lem}
\label{l1c}
 Consider a ball
$$\mathcal{B}_R=\{W(\cdot) \in {\mathbf W}^{1,1}([0,T],\mathcal{F}_{\hat{\mathcal K}})| \ \|W(\cdot) \|_{{\mathbf W}^{1,1}} \leq R\}.$$
Then
$$\exists L_R>0: \  \|u^*_2(t)-u^*_1(t)\|_H \leq L_R \int_0^T \|W_2(t)-W_1(t)\|_{\mathcal{F}_{\hat{\mathcal K}}}dt, \ \forall t \in [0,T],$$
 $\forall W_1(\cdot),W_2 (\cdot)\in \mathcal{B}_R$ and the corresponding trajectories  $u^*_1(t),u^*_2(t)$ of (\ref{red}).
 This means that the input-trajectory map  $E^*$  is Lipschitzian on $\mathcal{B}_R$, endowed  with  ${\mathbf L}^1([0,T],\mathcal{F}_{\hat{\mathcal K}})$-metric,  the space of trajectories $u^*(\cdot)$  endowed with $C([0,T],H^s)$-metric.  $\Box$
\end{lem}

We postpone the proof of Lemma~\ref{l1c} to  the Appendix, and now derive the  Main Result 2.

Consider the composition of maps
$$(\dot{w}_k)_{k \in \hat{\mathcal K}} \mapsto W(\cdot) \mapsto E_T^*(W)=E^*(W)|_{t=T}, $$
 where $E_T^*$ is the end-point map $W(\cdot) \mapsto u|_{t=T}$ for the equation (\ref{red}).

The relation between  the end-point maps of the controlled equations (\ref{nlsdot}) and (\ref{red}) provides the equality
$E_T((\dot{w})= E_T^*(W(\cdot))+iW(T)$ and therefore the image of $E_T$ (the attainable set) is contained in the image of the map
$$\Theta: (W(\cdot),\vartheta) \mapsto E_T^*(W(\cdot))+\vartheta, \ (W(\cdot),\vartheta) \in {\mathbf W}^{1,1}([0,T],\mathcal{F}_{\hat{\mathcal K}}) \times \mathbb{C}^\kappa . $$

Representing  ${\mathbf L}^1([0,T],\mathbb{C}^\kappa)$ as a union of the balls $\bigcup_{n \geq 1} \mathcal{B}_n$ of radii $n \in \mathbb{N}$ we note that the image  $\mathcal{I}(\mathcal{B}_n)$  under the map $\mathcal{I}: \dot{w}(\cdot) \mapsto (w(\cdot),w(T))$ is bounded
in ${\mathbf W}^{1,1}([0,T],\mathcal{F}_{\hat{\mathcal K}}) \times \mathbb{C}^\kappa$ and  pre-compact (completely bounded)  in  ${\mathbf W}^{1,1}([0,T],\mathcal{F}_{\hat{\mathcal K}}) \times \mathbb{C}^\kappa$, endowed  with the metric of  ${\mathbf L}^{1}([0,T],\mathcal{F}_{\hat{\mathcal K}}) \times \mathbb{C}^\kappa$.

By the Lemma~\ref{l1c} the map $E_T^*$
is Lipschitzian  in the metric of ${\mathbf L}^{1}([0,T],\mathcal{F}_{\hat{\mathcal K}})$;  therefore  $\Theta$ is  Lipschitzian  in the metric of  ${\mathbf L}^{1}([0,T],\mathcal{F}_{\hat{\mathcal K}}) \times \mathbb{C}^\kappa$ and hence  $E_T(\mathcal{B}_n)$ is contained in  completely bounded  image
$\Theta (\mathcal{I}(\mathcal{B}_n)) \subset H^s$.  Then  the attainable set of (\ref{nlsdot}) is contained in a countable union of pre-compacts $\bigcup_{n \geq 1} \Theta (\mathcal{I}(B_n))$ and by Baire category theorem has a dense complement in $H^s. \ \Box$

\section{Saturating sets of controlled modes and controllability}
\label{sscm}

As we  have established above  the saturating sets of controlled modes suffice for providing  the approximate controllability and controllability in projection on each finite-di\-men\-si\-onal subspace for the NLS equation \eqref{nls1}.

Here  we describe some   classes
of saturating sets in $\mathbb{Z}^d, \ d \geq 1$.

Starting with a set $\hat{\mathcal K} \subset \mathbb{Z}^d$ and recalling the  definition of proper   extension we define the
 sequence of sets $\mathcal{K}^j \subset \mathbb{Z}^d$:
\begin{equation}\label{setk2}
\mathcal{K}^1=\hat{\mathcal K}, \mathcal{K}^j  =  \left\{2m-n | \ m,n \in
\mathcal{K}^{j-1} \right\}, \  j=2, \ldots  ; \  \mathcal{K}^\infty =\bigcup_{j=1}^\infty \mathcal{K}^j.
\end{equation}
Taking $m=n$ in (\ref{setk2}) we conclude that $\mathcal{K}^{1} \subseteq \cdots \subseteq \mathcal{K}^{j} \subseteq  \cdots \subseteq \mathcal{K}^\infty $ and for each pair $j<j'$ the set $\mathcal{K}^{j'}$ is proper extension of $\mathcal{K}^{j}$. The set $\mathcal{K}^\infty $ is invariant with respect to the operation
\begin{equation}\label{oper}
(k,\ell) \mapsto 2k - \ell, \ k,\ell \in \mathbb{Z}^d.
\end{equation}

 A finite set $\hat{\mathcal K} \subset \mathbb{Z}^d$
of modes is  saturating if and only if
$\mathcal{K}^\infty=\mathbb{Z}^d.$

\begin{proposition}
\label{z1}
For each $k \in \mathbb{Z}$ the  two-element set $\hat{\mathcal{K}}=\{k,k+1\}$ is saturating in $\mathbb{Z}.  \ \Box$
\end{proposition}

\begin{proof}
Note that $k-1=2k-(k+1) \in \mathcal{K}^\infty$.

Assume that  for some natural $N$: $k-N,k-N+1, \ldots , k+N \in  \mathcal{K}^\infty$. Then
\begin{eqnarray*}
2(k + 1)-(k - (N-1))=k + (N+1) \in \mathcal{K}^\infty,  \\
  2(k - 1)-(k + (N-1))=k - (N+1) \in \mathcal{K}^\infty ,
\end{eqnarray*}
and the needed conclusion follows by induction.
\end{proof}

\begin{proposition}
\label{satu}
Let the vectors $k_1, \cdots , k_d \in \mathbb{Z}^d$ be such that
\begin{equation}\label{wd}
D=\det\left(k_1 , k_2 , \cdots , k_d\right)= \pm 1.
\end{equation}
 Then the $2^d$-element set
\begin{equation}\label{satd}
\hat{\mathcal{K}}^d=\left\{\left.\sum_{j=1}^pk_{i_j}\right| \ i_1 < \cdots < i_p, \ 0 \leq p \leq d \right\} \subset \mathbb{Z}^d    \end{equation}
  is saturating. $\Box$
\end{proposition}
\begin{remark}
In \eqref{satd} $p=0$ corresponds to the null vector $\mathbf{0} \in \mathbb{Z}^d. \ \Box$
\end{remark}

\begin{proof} i) Note that if $z \in \mathcal{K}^\infty$, then $-z=2 \cdot \mathbf{0}-z  \in   \mathcal{K}^\infty$.

ii) We prove first  that  for {\it any} set  $\hat{\mathcal{K}}=\{\mathbf{0},\ell_1, \ldots ,\ell_N\}$ the set     $\mathcal{K}^\infty$, defined by \eqref{setk2} coincides with the set $C^{N}_{e}$ of all integer linear combinations
$\sum_{s=1}^N \alpha_s \ell_s$ with at most one  odd coefficient $\alpha_s$.

Indeed $C^{N}_{e} \supset \hat{\mathcal{K}}$ and   is  invariant with respect to the operation \eqref{oper}.

 Also $C^{N}_{e} \supset -\hat{\mathcal{K}}$  and if  $\pm z \in \mathcal{K}^\infty$, then $\forall s_0: \  \pm z +2\ell_{s_0} \in \mathcal{K}^\infty$ and
 by induction $\pm z + \sum_{s=1}^N 2\alpha_s \ell_s \in   \mathcal{K}^\infty$, for any   $\alpha_1 , \ldots , \alpha_N  \in \mathbb{Z}$. In particular each  linear combination $\sum_{s=1}^N 2\alpha_s \ell_s$ with even coefficients belongs to $\mathcal{K}^\infty$, together with  $\ell_{s_0}+\sum_{s=1}^N 2\alpha_s \ell_s \in   \mathcal{K}^\infty$ for each $s_0$.
Hence $C^{N}_{e}=\mathcal{K}^\infty$.

iii) Now we prove that for $\hat{\mathcal K}^d$ defined by
\eqref{satd} the extension $\mathcal{K}^\infty$ contains all integer linear combinations $\sum_{j=1}^d\alpha_jk_j$.

Picking any combination  $\sum_{j=1}^d\alpha_jk_j$ we select  its odd coefficients
$\alpha_{j_1}, \ldots , \alpha_{j_p}$,  $j_1 < \cdots < j_p$. Then
$$\sum_{j=1}^d\alpha_jk_j= \left(k_{j_1}+\cdots + k_{j_p}\right)+\sum_{j=1}^d\alpha'_jk_j,$$
with all $\alpha'_j$ even. Hence $\sum_{j=1}^d\alpha_jk_j$ is representable as a linear combination of the elements of  \eqref{satd} with at most  one odd coefficient.
By ii) this proves that $\sum_{j=1}^d\alpha_jk_j \in \mathcal{K}^\infty$.

iv) Finally we prove that whenever \eqref{wd} holds,   the set  of all integer combinations  $\sum_{j=1}^d\alpha_jk_j$  coincides with $\mathbb{Z}^d$.

  Picking any $\ell \in \mathbb{Z}^d$, we  solve the equation
  $\sum_{j=1}^d\alpha_jk_j=\ell $ by Kramer rule
  computing  $\alpha_j=D_j/D, \ j=1, \ldots , d,$ where $D_j=\det(k_1, \ldots , \underbrace{\ell}_j, \ldots k_d)$ are integer, and $D=\pm 1$ according to \eqref{wd}.
  \end{proof}

\begin{remark}\label{k111}
An example of a saturating set  $\hat{\mathcal K} \subset \mathbb{Z}^d$  is the set
 of all vectors with components equal either $0$ or $1$ or, in other words, the set of vertices of the unit cube $[0,1]^n.  \ \Box$
\end{remark}


\section{Appendix: some proofs}

\subsection{Proof of the Lemma~\ref{limeq}}
\begin{proof}
Similarly to   (\ref{umut}) we get an estimate (all the norms are taken in $H$):
\begin{eqnarray*}
 \|u^\eps(t)-\tilde{u}(t)\| \leq   \left\|e^{i\eps t \Delta}u^0-u^0 \right\|+\left\|\int_0^te^{-i\eps\tau \Delta}\eps G(\tau,u^\eps(\tau),b)d\tau\right\|+\\
  +\left\|\int_0^t e^{-i\eps\tau \Delta}\phi(\tau , u^\eps(\tau),b)-\phi(\tau,\tilde{u}(\tau),b)d\tau\right\| \leq  \\
  \leq \eps\left\|\int_0^te^{-i\eps\tau \Delta}G(\tau,u^\eps (\tau),b)d\tau\right\|   + \left\|(e^{i\eps t \Delta}-I)u^0\right\|  + \\ +
     \left\|\int_0^t\!\!(e^{-i\eps \tau \Delta}-I)\phi(\tau,\tilde{u}(\tau),b)d\tau \right\| +\int_0^t \!\! \left\|\phi(\tau , u^\eps(\tau),b)-\phi(\tau,\tilde{u}(\tau),b)\right\| d\tau.
\end{eqnarray*}
The rightmost addend admits an upper bound
$\int_0^t \beta_c(\tau)\|u^\eps(\tau)-\tilde{u}(\tau)\| d\tau $.  We would arrive to the needed conclusion  by virtue of  Gronwall inequality,  after proving that the other three addends tend to $0$, as $\eps \to 0$.

We only comment on  the addend $\left\|\int_0^t(e^{-i\eps \tau \Delta}-I)\phi(\tau,\tilde{u}(\tau),b)d\tau\right\|$, the other two assertions being obvious.

    By the assumptions of the Lemma the map $(\tau,b) \mapsto \phi(\tau , \tilde{u}(\tau),b)$
    is continuous on $[0,T] \times B$ and hence its range $R$ is a compact subset of Hilbert space $H$.
    Being $\pi_N$ the orthogonal projection of $H$ on $\mathcal{F}_N$ and  $R_N=\pi_N(R)$, we assert that
    $$\forall \eps >0 \exists N, \forall n \geq N: \ \mbox{dist}(R,R_N) < \eps/2. $$
    Also there exists $c_N>0$ such that $$\forall y \in R_N, \ \forall t \in [0,T]: \
    \left\|\left(e^{-i\eps t \Delta} -I\right)y\right\| \leq c_N \eps .$$
    Given that $\|e^{-i\eps \tau \Delta}-I\| \leq 2$ one concludes
    \begin{eqnarray*}
      \forall z \in R: \ \left\|\left(e^{-i\eps t \Delta} -I\right) z\right\| \leq \left\|\left(e^{-i\eps t \Delta} -I\right)(z- \pi_N z) \right\|+ \\
       +\left\|\left(e^{-i\eps t \Delta} -I\right) \pi_N z\right\| \leq 2 (\eps/2)+c_N\eps =(c_N+1)\eps .
    \end{eqnarray*}
    and
$$\left\|\int_0^t(e^{-i\eps \tau \Delta}-I)\phi(\tau,\tilde{u}(\tau),b)d\tau \right\| \leq (c_N+1)\eps T.$$
\end{proof}

\subsection{Proof of the Lemma~\ref{5.4}.}
In the proof we use $\eps$ in place of  $\eps' $.

  As in the Remark~\ref{smow} we can smoothen the family $w_{2r-s}(t,b)$  arriving to
  a family of smooth functions $\check{w}(t,b)$ such that $\forall b: \ \|\check{w}(t,b)-w_{2r-s}\|_{{\mathbf L}^1_t} \leq \eps /2$, and $\check{w}(t,b),\partial_t \check{w}(t,b)$ depend on $b$ continuously in ${\mathbf L}^1_t$-metric. This implies that $\|\check{w}(t,b)\|_{{\mathbf L}^\infty}$ are equibounded  by some $C_w>0$, and for each $\tau \in [0,T]$ the values $\check{w}(\tau ,b)$ depend continuously on $b$.

      First we choose real-valued nonnegative continuous function $\check{v}_r(t)$ such that $\check{v}^2_r(t)$  is:   piecewise-linear, vanishing at $0,T$,  and (constant) equal to $\pi(T-\varepsilon^2)^{-1}$ on $[\varepsilon^2 ,T-\varepsilon^2]$, while being  linear on $[0,\varepsilon^2]$ and $[T-\varepsilon^2 , T]$.

      By the construction  $\int_0^T\check{v}^2_r(t)dt=\pi$.
      Assuming $\eps^2 < T/2$, we get  $\|\check{v}^2_r\|_{L^\infty} \leq 2\pi/T$.

Denote $\mathcal{I}_\eps=[0,\eps^2] \bigcup [T-\eps^2,T]$ and let  $w^\eps(t,b)$ be the restriction of the function $\pi^{-1}(T-\eps^2)\check{w}(t,b)$ onto $[0,T]\setminus \mathcal{I}_\eps$ .  Let $$\int_{\eps^2}^{T-\eps^2}|w^\eps(t,b)|^2dt=A(b), \ A=\max_{b \in B} A(b);$$
   $A(b)$ is continuous and the maximum is achieved.
    Take $N=[A/\pi]+1$. Extend
     \footnote{For example by polynomials of fixed degree, defined on $[0,\eps^2]$ and $[T-\eps^2]$, with coefficients depending continuously on $b$}
     the family $|w^\varepsilon(t,b)|$ to  a family of Lipschitzian functions $\check{v}_s(t,b)$ on $[0,T]$ such  that $\check{v}_s(0,b)=\check{v}_s(T,b)=0$, $\int_0^T|\check{v}_s(t,b)|^2dt=\pi N$
     and $\check{v}_s(t,b)$, $\partial_t\check{v}_s(t,b)$ depend continuously on $b$ in ${\mathbf L}^1_t$-metric.

Then
$$ \int_0^T|\check{v}_r(t)|^2+|\check{v}_s(t,b)|^2dt=\pi (N+1). $$

Obviously $\check{v}_r^2(t;b)\check{v}_s(t;b)=|w_{2r-s}(t,b)|$ on $[\eps^2,T-\eps^2]$ and
$$\int_{\mathcal{I}_\eps} |\check{v}_s(t)|^2dt \leq \pi N ;$$
hence $\int_{\mathcal{I}_\eps}|\check{v}_s(t;b)|dt  \leq  \varepsilon \sqrt{2\pi N} $ by Cauchy-Schwarz inequality.

Then
\begin{eqnarray*}
 \int_0^T\left|\check{v}^2_r(t)\check{v}_s(t,b)-|w_{2r-s}(t,b)|\right|dt \leq
\frac{\eps}{2}+ \\ + \int_0^T\left|\check{v}^2_r(t)\check{v}_s(t,b)-|\check{w}(t,b)|\right|dt \leq \eps/2+  \int_{\mathcal{I}_\eps}\left(\left|\check{v}^2_r(t)\check{v}_s(t,b)\right|+ |\check{w}(t;b)|\right)dt \leq \\  \leq \frac{\eps}{2}+   \|\check{v}^2_r(t)\|_{{\mathbf L}^\infty}\varepsilon \sqrt{2\pi N}+2C_w \eps^2 \leq
\frac{\eps}{2}+(2\pi \eps /T)\sqrt{2\pi N}+2C_w \eps^2  .  \Box
\end{eqnarray*}

\subsection{Proof of Lemma~\ref{coconv}.}
For a continuous  $\tilde{u}(t)$  and $\Phi$ possessing Lipschitzian property,
we prove  that $\forall \delta >0 \ \exists \delta '>0$ such that
$\forall \phi \in \Phi$:
\begin{equation}\label{tilu}
  \sup_{t,t' \in[0,T],\|u\|\leq b}\left\|
\int\limits_t^{t'} \phi(\tau,u)d \tau \right\| < \delta ' \Rightarrow \\   \sup_{t,t' \in[0,T]}\left\|
\int\limits_t^{t'} \phi(\tau , \tilde{u}(\tau))d \tau \right\| < \delta .
\end{equation}

Indeed if $\omega(\tau)$ is the modulus of continuity for $\tilde{u}(t)$ and $\sup_{t \in [0,T]}\|\tilde{u}(t)\| \leq c$, then
\begin{equation*}
 \sup_{t,t' \in[0,T]}\left\|
\int\limits_t^{t'} \phi(\tau , \tilde{u}(\tau))d \tau \right\|=\left\|
\int\limits_{\underline{t}}^{\underline{t}'} \phi(\tau , \tilde{u}(\tau))d \tau \right\| \leq
\sum_{j=0}^{N-1}\left\|\int\limits_{t_j}^{t_{j+1}} \phi(\tau , \tilde{u}(\tau))d \tau \right\|,
\end{equation*}
where $\underline{t}=t_0 < t_1 < \cdots < t_N=\underline{t}'$ is a partition of $[\underline{t},\underline{t}'] \subset [0,T]$ into $N \leq T/\eta$ subintervals of length $\eta$. Then
 \begin{eqnarray*}
    \sup_{t,t' \in[0,T]}\left\|
\int\limits_t^{t'} \phi(\tau , \tilde{u}(\tau))d \tau \right\| \leq  \sum_{j=0}^{N-1}\left\|\int\limits_{t_j}^{t_{j+1}} \left(\phi(\tau , \tilde{u}(\tau))- \phi(\tau , \tilde{u}(t_j))\right)d \tau \right\|+\\
+ \sum_{j=0}^{N-1}\left\|\int\limits_{t_j}^{t_{j+1}} \phi(\tau , \tilde{u}(t_j))d \tau \right\| \leq
\sum_{j=0}^{N-1} \int\limits_{t_j}^{t_{j+1}} \beta_c(\tau)\|\tilde{u}(\tau))- \tilde{u}(t_j)\| d \tau + \\  +N\|\phi\|_{rx} \leq
 C\omega(\eta)+\frac{T}{\eta}\|\phi\|_{rx}.
 \end{eqnarray*}
Choosing $\eta=\|\phi\|_{rx}^{1/2}$ we get
$$\sup_{t,t' \in[0,T]}\left\|
\int\limits_t^{t'} \phi(\tau , \tilde{u}(\tau))d \tau \right\| \leq C\omega(\|\phi\|_{rx}^{1/2})+T\|\phi\|_{rx}^{1/2} $$
and conclude (\ref{tilu}).

Let us introduce
$$ \ \tilde{\Phi}=\{\varphi(\cdot)| \ \varphi(\cdot)=\phi(\cdot,\tilde{u}(\cdot)),\ \phi \in \Phi\}. $$
According to the aforesaid it suffices to prove
that $\forall \eps>0 \exists \delta >0$ such that $\forall \varphi \in \tilde{\Phi}$:
\begin{equation}\label{cocu}
        \|\varphi\|^{rx}< \delta \Rightarrow   \left\|\int_0^t e^{-i\tau \Delta}\varphi(\tau)d\tau \right\| < \eps .
\end{equation}

The set $R=\{\varphi(\tau)| \ \tau \in [0,T], \varphi \in \tilde{\Phi}\} $
is completely bounded by the assumptions of the Lemma.

                Taking for a natural $n$ the orthogonal projection $\Pi_n$ onto $\mathcal{F}_n=\mbox{Span}\{f_k| \ |k| \leq n \}$, we  conclude by   precompactness of $R$  that
$\sup_{x \in R}\|x-\Pi_n x\| \rightarrow 0$, as $n \rightarrow \infty.$

Take a partition $0=\tau_0 < \tau_1 < \cdots < \tau_N=T$ of the interval $[0,T]$ into subintervals of lengths $\eta=T/N$. We represent the integral in  (\ref{cocu}) as a sum
\begin{eqnarray*}
\int_0^t e^{-i\tau \Delta}\varphi(\tau)d\tau=\int_0^t e^{-i\tau \Delta}\left(\varphi(\tau)-\Pi_n\varphi(\tau)\right)d\tau +  \\
+ \sum_{j=1}^\omega e^{-i\tau_j \Delta} \int_{\tau_{j-1}}^{\tau_j} \Pi_n\varphi(\tau)d\tau
+ \\ + \sum_{j=0}^{N-1}  \int_{\tau_{j-1}}^{\tau_j}e^{-i\tau_j \Delta}\left(e^{-i(\tau-\tau_j) \Delta}- I \right) \Pi_n\varphi(\tau)d\tau .
\end{eqnarray*}

Recalling that:
  \begin{description}
          \item[i] $e^{-i\tau \Delta}$ is an isometry of $H$;
          \item[ii] $\left\|\int_{\tau_{j-1}}^{\tau_j} \Pi_n\varphi(\tau)d\tau \right\| \leq \|\varphi \|^{rx};$
          \item[iii] $\|\left(\varphi(\tau)-\Pi_n\varphi(\tau)\right)\| \leq \rho_n$, and $\rho_n \stackrel{n \rightarrow \infty}{\longrightarrow} 0$ uniformly for  $(\tau ,\varphi) \in [0,T] \times \tilde{\Phi}$;
       \item[iv]  $\sup_{0 \leq \xi \leq \tau}\|\left(e^{-i\xi \Delta}- I \right)\circ \Pi_n\|=\gamma_n(\tau)$, $\forall n: \ \lim_{\tau \rightarrow 0} \gamma_n(\tau) = 0$,
        \end{description}
   we conclude
   \begin{equation}\label{intri}
\left\|\int_0^t e^{-i\tau \Delta}\varphi(\tau)d\tau \right\| \leq T\rho_n+ T\eta^{-1} \|\varphi \|^{rx}+\gamma_n(\eta)\int_0^T\|\varphi(\tau)\|d\tau .
   \end{equation}

Recall that $\int_0^T\|\varphi(\tau)\|d\tau $ are bounded by a constant $c_1$ for all $\varphi \in \tilde{\Phi}$.

Taking  $n$ large enough so that $T \rho_n < \eps/3$, we then choose $\eta>0$
small enough so that  $c_1\gamma_n(\eta) < \eps/3$. If we impose
$\|\varphi \|^{rx} <\eps\eta/3T$, then (\ref{intri}) will imply
$\left\|\int_0^t e^{-i\tau \Delta}\varphi(\tau)d\tau \right\| < \eps . \ \Box$

\subsection{Proof of Lemma~\ref{l1c}.} By the inequalities (\ref{gro54})-(\ref{55}) we get
$$\|u^*_2(t)-u^*_1(t)\|_H \leq C\int_0^T \|\Phi_{12}(\tau , u^*_1(\tau))\|_Hd\tau e^{C'\int_0^T\beta_c(\tau)d\tau},$$
where
$$\Phi_{12}(\tau , u)=|u+iW_1(t)|^2(u+iW_1(t))-|u+iW_2(t)|^2(u+iW_2(t)),$$
while  $\beta_c(t)$ characterizes Lipschitzian property:
\begin{eqnarray*}
\| |u'+iW(t)|^2(u'+iW(t))-|u+iW(t)|^2(u+iW(t)) \|_H \leq      \nonumber \\
\leq \beta_c(t)\|u'-u \|_H, \   \forall W(\cdot) \in B_R, \ \|u\|_{H} \leq c. 
\end{eqnarray*}

By the Product Lemma~\ref{prol}
 $\beta_c$ can be chosen constant, equal to $C'(1+c^2+R^2)$, as far as  $W(t)$ are trigonometric polynomials in $x$  and their norms  in ${\mathbf W}^{1,1}([0,T], \mathcal{F}_{\hat{\mathcal{K}}}$ are equibounded.

Similarly
  $$\|\Phi_{12}(\tau , u^*_1(\tau))\|_H \leq C_1(1+c^2+R^2)\|W_2(\tau)-W_1(\tau)\|_H.$$
  Then for $L_R=CC_1(1+c^2+R^2)e^{C'(1+c^2+R^2)T}$:
  $$\|u^*_2(t)-u^*_1(t)\|_H \leq L_R\!\! \int_0^T \!\!\! \|W_2(\tau)-W_1(\tau) \|d\tau . \ \Box$$


\end{document}